\documentclass[12pt,reqno,oneside]{amsart}
\usepackage{amsmath}
\usepackage{amssymb}
\usepackage{amsthm}
\usepackage{graphics}
\usepackage{eucal}
\usepackage{mathrsfs}
\usepackage[pdftex,bookmarks,colorlinks,breaklinks]{hyperref}

\newcommand{\vertiii}[1]{{\left\vert\kern-0.25ex\left\vert\kern-0.25ex\left\vert #1
    \right\vert\kern-0.25ex\right\vert\kern-0.25ex\right\vert}}
\theoremstyle{plain}
\usepackage{geometry}

\newtheorem{theorem}{Theorem}[subsection]
\newtheorem{definition}{Definition}[section]
\newtheorem{proposition}{Proposition}[subsection]
\newtheorem{lemma}{Lemma}[subsection]
\newtheorem{corollary}{Corollary}[subsection]

\newtheorem{remark}{Remark}[subsection]

\numberwithin{equation}{section}

\newcommand{\diam}{\operatorname{diam}}
\newcommand{\dist}{\operatorname{dist}}

\newcommand{\supp}{\operatorname{sup}}
\newcommand{\bbR}{\mathbb{R}}
\newcommand{\bbQ}{\mathbb{Q}}
\newcommand{\bbZ}{\mathbb{Z}}
\newcommand{\bt}{\mathbf{t}}

\newcommand{\bx}{\mathbf{x}}
\newcommand{\bU}{\mathbf{U}}
\newcommand{\bB}{\mathbf{B}}
\newcommand{\cR}{\mathcal{R}}

\newcommand{\f}{\mathbf{f}}

\newcommand{\bz}{\mathbf{z}}

\newcommand{\ba}{\mathbf{a}}

\newcommand{\Vol}{\operatorname{Vol}}

\newcommand{\bc}{\mathbf{c}}

\newcommand{\by}{\mathbf{y}}

\newcommand{\al}{\alpha}

\newcommand{\cW}{\mathcal{W}}

\newcommand{\Q}{\mathbb{Q}}
\newcommand{\Z}{\mathbb{Z}}
\newcommand{\R}{\mathbb{R}}
\newcommand{\N}{\mathbb{N}}
\newcommand{\bT}{\mathbf{T}}
\newcommand{\bV}{\mathbf{V}}
\newcommand{\cA}{\mathcal{A}}
\newcommand{\bS}{\mathbf{S}}

\newcommand{\fnu}{f_\nu^{(1)},f_\nu^{(2)},\dots,f_\nu^{(n)}}
\newcommand{\well}{\cW_{\Psi,\Theta}^{\f}}
\newcommand{\bW}{\mathbf{W}}

\newcommand{\Wf}{\bW_\f(\ba,\Psi,\Theta)}
\newcommand{\Wfl}{\bW_{\f}^{\text{large}}(\ba,\Psi,\Theta)}
\newcommand{\Wfsnu}{\bW_{\nu,\f}^{\text{small}}(\ba,\Psi,\Theta)}
\newcommand{\welllarge}{\limsup_{\ba\to\infty}\Wfl}
\newcommand{\wellsmallnu}{\limsup_{\ba\to\infty}\Wfsnu}
\newcommand{\inv}{^{\text{-}1}}
\newcommand{\overbar}[1]{\mkern 1.5mu\overline{\mkern-1.5mu#1\mkern-1.5mu}\mkern 1.5mu}
\begin{document}
\title[Inhomogeneous Diophantine approximation]{$S$-arithmetic  Inhomogeneous Diophantine approximation on manifolds}

\begin{abstract}
We prove $S$-arithmetic inhomogeneous Khintchine type theorems on analytic nondegenerate manifolds. The divergence case, which constitutes the main substance of this paper, is proved in the general context of Hausdorff measures using ubiquitous systems. For $S$ consisting of more than one prime, finite or infinite, the divergence results are new even in the homogeneous setting. We also prove the convergence case of the theorem, including in particular, the $S$-arithmetic inhomogeneous counterpart of the Baker-Sprind\v{z}uk conjectures using nondivergence estimates for flows on homogeneous spaces in conjunction with the transference principle of Beresnevich and Velani.  
\end{abstract}
\author{Shreyasi Datta}
\author{Anish Ghosh}

\address{School of Mathematics, Tata Institute of Fundamental Research, Mumbai, 400005, India}
\email{shreya@math.tifr.res.in, ghosh@math.tifr.res.in}
\thanks{AG was supported by a Government of India, Department of Science and Technology, Swarnajayanti fellowship DST/SJF/MSA-01/2016-17, a CEFIPRA grant, a MATRICS grant and a UGC grant. This work was supported by a grant from the Infosys foundation. The authors  acknowledge support of the Department of Atomic Energy, Government of India, under project $12-R\&D-TFR-5.01-0500$.}
\maketitle
\tableofcontents
\section{Introduction}




The context of this paper is $p$-adic Diophantine approximation, or more generally, $S$-arithmetic Diophantine approximation for a finite set of primes $S$. Let $\Psi : \bbR^n \to \bbR_{+} $ be a function satisfying
\begin{equation}\label{defmultapp}
\Psi(a_1, \dots, a_n) \geq \Psi(b_1, \dots, b_n) \text{ if } |a_i| \leq |b_i| \text{ for all } i = 1,\dots, n.
\end{equation}
Such a function is referred to as a \emph{multivariable approximating function}.
We follow the notation of Kleinbock and Tomanov \cite{KT} in setting up the basics of $S$-arithmetic Diophantine approximation.  Given a finite set of primes $S$ of cardinality $l$ we set $\Q_S := \prod_{\nu \in S}\Q_\nu$, denote by $|~|_S$ the $S$-adic absolute value, 
$$|x|_S = \max_{v \in S }|x^{(v)}|_v$$ and set $$\Vert\cdot\Vert_S=\max_{v\in S}\Vert \cdot\Vert_v,$$ to be the supremum norm in $\Q_S^n$. Sometimes we will drop the suffixes of norms if it is clear from the context. For $\ba = (a_1, \dots, a_n) \in \bbZ^n$ and $a_0 \in \bbZ$ we set $ \widetilde{\ba} := (a_0, a_1, \dots, a_n).$ We wish to consider an $S$-arithmetic version of Khintchine's theorem, the foundational result in metric Diophantine approximation. It turns out that the definition of $\Psi$-approximable vectors in the $S$-arithmetic setting depends on whether or not $S$ contains the infinite place, we refer the reader to the discussion in \S 11 of \cite{KT} for the reason, which has to do with the interplay between Archimedean and ultrametric norms.  We first consider the case that $\infty \in S$. 
For a function $\Psi$ as in (\ref{defmultapp}) we say that $\by \in \Q^{n}_S$ is $\Psi$-approximable ($\by \in \cW_{n}(S, \Psi)$) if there are infinitely many solutions $\ba \in \Z^n$ to
\begin{equation}\label{def:with} 
|a_0 + \ba\cdot \by|_{S}^{l}  \leq \Psi(\ba). 
\end{equation}

For the case $\infty \notin S$, we consider a function $\Psi :  \bbR^{n+1} \to \bbR_{+} $ satisfying (\ref{defmultapp}), i.e. $\Psi$ is a multivariable approximating function in $n+1$ variables. We say that $\by \in \Q^{n}_S$ is $\Psi$-approximable if there are infinitely many solutions $\ba \in \Z^n$ to
\begin{equation}\label{def:without} 
|a_0 + \ba\cdot \by|_{S}^{l}  \leq \Psi(\widetilde{\ba}).  
\end{equation}

\noindent We fix Haar measure on $\Q_p$, normalized to give $\Z_p$ measure $1$ and denote the product measure on $\Q_S$ by $|~|_S$. Then, as $S$-arithmetic analogue of Khintchine's theorem in Diophantine approximation can be proved using the methods in \cite{Lutz}. In this paper, we are concerned with $S$-arithmetic analogues of the theorem and its inhomogeneous counterparts, in the setting of manifolds. We will study both $\infty \in S$ and $\infty \notin S$ cases in this paper, and so we will adopt the convention for the rest of the paper that the approximating function $\Psi$ will be understood to be a function of $n$ or $n+1$ variables according as  $\infty \in S$ or $\infty \notin S$.                          

One can similarly setup the corresponding inhomogeneous problem. Here we follow \cite{BaBeVe} where the notation below was introduced. For a multivariable approximating function $\Psi$ and a function $\Theta: \bbQ^{n}_S \to \bbQ_S$,  we say that a vector  $\bx  \in  \bbQ_S^n $ is $(\Psi,\Theta)$-approximable if there exist infinitely many  $(\ba, a_0)\in\bbZ^n\setminus\{0\}\times \bbZ $ such that  
 \begin{equation}
 |a_0 + \ba\cdot \bx +\Theta(\bx)|_{S}^l\leq  \left\{
\begin{array}{rl} 
\Psi(\widetilde{\ba}) & \text{ if } \infty \notin S\\ 
\Psi(\ba) & \text{ if } \infty \in S.
\end{array} \right.
\end{equation}

The convergence case of Khintchine's theorem in this setting again follows from the Borel Cantelli lemma. The divergence Theorem for a slightly more restrictive class of functions, as explicated below, is a special case of the results in the present paper.\\  

In fact our results are much more general; we establish a complete Khintchine type theory, both homogeneous and inhomogeneous, for nondegenerate manifolds in the $S$-arithmetic setting. The main results proved in this paper are:
\begin{enumerate}
\item We prove the divergence Khintchine theorem for nondegenerate manifolds in the setting of Hausdorff measures, both in the homogeneous as well as inhomogeneous contexts. We would like to emphasize that our results are new even in the homogeneous setting. Previously, the most general result was established in \cite{MoS2}, where the homogeneous divergence theorem was proved for $S$ comprising a single prime. The cases where $S$ contains several finite primes and where $S$ contains several finite primes and the infinite prime are both treated here for the first time and need significant new ideas in order to apply the ubiquity framework. These results, namely Theorems \ref{thm:divergence} and \ref{thm:divergence_infty} below constitute the main results in this paper. The possibility of proving the homogeneous divergence Khintchine theorem for $S$ comprising more than one prime was raised by Mohammadi and Salehi-Golsefidy in \cite{MoS2}, \S 9.
\item We prove the convergence Khintchine theorem for nondegenerate manifolds in the inhomogeneous setting. This result, Theorem \ref{thm:main} below is a result of combining the homogeneous theorem in \cite{MoS1} with a suitable adaptation of the transference principles in \cite{BaBeVe}. The question of proving inhomogeneous Khintchine theorems for convergence and divergence was raised by Badziahin, Beresnevich and Velani in \cite{BaBeVe}.
\item As mentioned above, our divergence Theorems \ref{thm:divergence} and \ref{thm:divergence_infty} are, as far as we are aware, new even in the classical setting, namely for $\bbQ^{n}_S$. Applying the mass transference principle of Beresnevich and Velani gives one side (lower bounds) of an $S$-arithmetic Jarnik-Besicovitch theorem. We will pursue an $S$-arithmetic Hausdorff measure and dimension theory systematically elsewhere.

\end{enumerate}

Before stating our main results, we briefly review the state of the art in Diophantine approximation on manifolds, the $S$-arithmetic theory, and inhomogeneous Diophantine approximation. 
 
 \subsection{Diophantine approximation on manifolds}
 
 In the theory of Diophantine approximation on manifolds, one studies the inheritance of generic (for Lebesgue measure) Diophantine properties by proper submanifolds of $\R^n$. 
 Motivated by problems in transcendental number theory, K. Mahler conjectured in 1932 that almost every point on the curve $(x, x^2, \dots, x^n)$ is not very well approximable. 
 This conjecture was resolved by V. G. Sprind\v{z}uk \cite{Sp, Sp3} who in turn conjectured that almost every point on an analytic nondegenerate (loosely speaking, not locally contained in an affine subspace) manifold is not very well approximable. This conjecture, in a more general, multiplicative form, was resolved by D. Kleinbock and G. Margulis in \cite{KM} following earlier work by several authors, we refer the reader to the book \cite{BerDod} for a comprehensive historical account as well as the relevant definitions. 
Subsequent to the work of Kleinbock and Margulis, there were rapid advances in the theory of dual approximation on manifolds. In \cite{BKM} (and independently in \cite{Ber1}) the convergence case of the Khintchine-Groshev theorem for nondegenerate manifolds  was proved and in \cite{BBKM}, the complementary divergence case was established. In \cite{Kleinbock-extremal}, analogues of the Baker-Sprind\v{z}uk conjectures were established for affine subspaces. Khintchine type theorems were established in \cite{BBDD, G1, G-Monat, G3, G4} and also \cite{GG-quant}), we refer the reader to \cite{G-handbook} for a survey of results. 

As for the $p$-adic theory, the classical theory was initiated by Lutz \cite{Lutz} and studied by several authors including Mahler \cite{Mahler}, see \cite{Haynes} for recent results in this direction. Sprind\v{z}uk \cite{Sp} himself established the $p$-adic and function field (i.e. positive characteristic) versions of Mahler's conjectures. Subsequently, there were several partial results (cf. \cite{Kov, BK}) culminating in the work of Kleinbock and Tomanov \cite{KT} where the $S$-adic case of the Baker-Sprind\v{z}uk conjectures were settled  in full generality. The convergence case of Khintchine's theorem for nondegenerate manifolds in the $S$-adic setting was established by Golsefidy and Mohammadi \cite{MoS1}; they proved the divergence case for $\Q_p$ in \cite{MoS2}. Recently, in \cite{DG1, DG2} the authors have answered questions of Kleinbock, in particular establshing the $p$-adic Baker-Sprind\v{z}uk conjectures for affine subspaces. There has been considerable activity in the $p$-adic theory of Diophantine approximation for manifolds, see for instance \cite{Ber2, BBD, BBK, BDY, BeK, BZ, U1, U2, Zelu}. In particular, homogeneous and inhomogeneous versions of Khintchine's theorem have been proved for polynomial curves in various settings in the above cited papers.

 In \cite{G}, the second named author established the function field analogue of the Baker-Sprind\v{z}uk conjectures, and in \cite{GG-inhom} the inhomogeneous version was proved (see also \cite{GG-scand}).

In the case of inhomogeneous Diophantine approximation on manifolds, following several partial results (cf. \cite{Bu} and the references in \cite{BeVe, BeVe2}), an inhomogeneous transference principle was developed by Beresnevich and Velani in an important work, using which they resolved the inhomogeneous analogue of the Baker-Sprind\v{z}uk conjectures. Subsequently, Badziahin, Beresnevich and Velani \cite{BaBeVe} established the convergence and divergence cases of the inhomogeneous Khintchine theorem for nondegenerate manifolds. They proved a new result even in the classical setting by allowing the inhomogeneous term to vary. The divergence theorem is established in the same paper in the more general setting of Hausdorff measures. Recently, in \cite{BGGV}, the inhomogeneous analogue of Khintchine's theorem for affine subspaces was established in both convergence and divergence cases, see also \cite{GM, CGGMS}.\\

 \subsection{Main Results}
 To state our main results, we introduce some notation following \cite{MoS1}, recall some of the assumptions from that paper and set forth one further standing assumption. The assumptions are as follows.
 
 \begin{enumerate}
                 
                   \item[(I1)]  We will consider the domain to be of the form $\bU=\prod_{\nu\in S} \bU_{\nu}  $ where $\bU_\nu\subset\Q_\nu^{d_\nu}  $ is an open box. Here, the norm is taken to be the Euclidean norm at the infinite place and the $L^{\infty}$ norm at finite places.  
                   \item[(I2)] We will consider functions $\f(\bx)  =(\f_\nu(x_\nu)) _{\nu\in S}$, for  $\bx=(x_\nu) \in\bU $ where 
                   $\f_\nu=(f_\nu^{(1)},f_\nu^{(2)},\dots,f_\nu^{(n)}): 
                  \bU_\nu\to \Q_\nu^n $ is an analytic map for any 
                   $\nu\in S $, and can be analytically extended to the boundary of $ \bU_\nu$. 
                   \item[(I3)] 
                   \subitem[(I3-a)] We assume that the restrictions of $1 ,\fnu $ to any open subset of $\bU_\nu $ are linearly independent over $\Q_\nu $. 
                   \subitem[(I3-b)] We assume that $\Vert\f(\bx)\Vert_S\leq 1,\|\nabla\f_\nu(x_\nu)\| \leq 1$ and $|\Phi_\beta \f_\nu(y_1,y_2,y_3)| \leq \frac{1}{2} $  for any $\nu \in S,$ second difference 
                   quotient $\Phi_\beta$ and $x_\nu,y_1,y_2,y_3 \in U_\nu$. Definitions are provided later in the paper. 
                  \item[(I4)]\label{monotone_cond} We assume that the function $\Psi :\bbZ^n \to \bbR_{+ } $ is monotone decreasing component wise
                   i.e. $$\Psi(a_1,\cdots,a_i,\cdots, a_n)\geq \Psi(a_1,\cdots, a'_{i},\cdots, a_n)$$ whenever $|a_i|\leq |a'_i| $.
                   \item[(I5)] We assume that $\Theta(\bx)=(\Theta_\nu(x_\nu)) $ where $\Theta :\bU \mapsto \Q_S $ is also analytic and can be extended analytically to the boundary of $\bU_\nu$. We will assume $\vert\Theta(\bx)\vert_S\leq 1,\|\nabla\Theta_\nu(x_\nu)\| \leq 1$ and $|\Phi_\beta \Theta_\nu(y_1,y_2,y_3)| \leq \frac{1}{2} $ for any $\nu \in S,$ second difference 
                   quotient $\Phi_\beta$ and $x_\nu,y_1,y_2,y_3 \in U_\nu$.

                    \end{enumerate}

We will denote by $\mathcal{H}^{s}(X) $ the $s$-dimensional Hausdorff measure of a subset $X$ of $\bbQ^{n}_{S}$, where $s > 0$ is a real number.
  \noindent  We can now state the  main Theorems of the present paper.               
 
  \begin{theorem}\label{thm:divergence}
 		 Assume that $\infty \notin S$. Suppose $\f:\bU\subset\prod\Q_\nu^{m_\nu}\to \Q_S^n$ satisfies (I2) and (I3) and $\bU=(\bU_\nu)$ is an open subset $\prod\Q_\nu^{m_\nu}.$ Let 
		\begin{equation}\label{def:newpsi}
		\Psi(\ba):= \psi(\|\ba\|_\infty), \ba\in\Z^{n+1} 
		\end{equation}
		 be an approximating function and assume that $s > \sum m_\nu -l=m-l$. Let $\Theta:\bU\to \Q_S$ be an analytic map satisfying (I5).  Then 
	\begin{equation}
 	\mathcal{H}^s(\mathcal{W}^\f_{(\Psi,\Theta)}\cap\bU)=\mathcal{H}^s(\bU) \text{    if   } \sum_{\ba \in \bbZ^{n+1} \backslash \{0\}} (\Psi(\ba))^{\frac{s+l-m}{l}}=\infty.  
 	\end{equation}
 	\end{theorem}
 When $S$ contains $\infty$, we have the following Theorem which is weaker than the previous Theorem  \ref{thm:divergence} where $S$ does not contain $\infty$.
  \begin{theorem}\label{thm:divergence_infty}
 	Assume that $\infty \in S$. Suppose $\f:\bU\subset\prod\Q_\nu^{m_\nu}\to \Q_S^n$ satisfies (I2) and (I3) and $\bU=(\bU_\nu)$ is an open subset $\prod\Q_\nu^{m_\nu}.$ Let 
 	\begin{equation}\label{def:newpsi}
 	\Psi(\ba) := \psi(\|\ba\|_\infty), \ba\in\Z^{n} 
 	\end{equation}
 	be an approximating function and assume that $s > \sum m_\nu -l=m-l$. Let $\Theta:\bU\to \Q_S$ be an analytic map satisfying (I5). Let us denote
 	 \begin{equation}
 	\widetilde {\mathcal{W}^\f_{(\Psi,\Theta)}}=\left\{\bx\in\bU~\left|~\begin{aligned} \exists~ \text{ infinitely many } \ba\in\Z^{n}\setminus 0, a_0\in \Z \text{ s.t }\\
 	\vert a_0+\ba\cdot\f^\nu(\bx_\nu)+\Theta^\nu(\bx_\nu)\vert_\nu^l\leq \Psi(\ba)~\forall \nu\in S\setminus\infty\\
 	\vert a_0+\ba\cdot\f^\infty(\bx_\infty)+\Theta^\infty(\bx_\infty)\vert_\infty^l\leq \Psi(\ba)\cdot \Vert\ba\Vert_\infty ^l
 	\end{aligned}\right\}.
\right. 	\end{equation} Then 
 	\begin{equation}
 	\mathcal{H}^s(\widetilde{\mathcal{W}^\f_{(\Psi,\Theta)}}\cap\bU)=\mathcal{H}^s(\bU) \text{    if   } \sum_{\ba \in \bbZ^{n} \backslash \{0\}} (\Psi(\ba))^{\frac{s+l-m}{l}}=\infty.  
 	\end{equation}
 \end{theorem}
\noindent Note that $\mathcal{W}^\f_{(\Psi,\Theta)}\subset\widetilde {\mathcal{W}^\f_{(\Psi,\Theta)}}$, hence this above theorem is weaker than Theorem \ref{thm:divergence}.
	
	
\noindent Our next Theorem is the convergence $S$-adic Khintchine type theorem.
 \begin{theorem}\label{thm:main}
Assume $\infty \in S$ and let $\bU$ as in (I1). Suppose $\f$ satisfies (I2) and (I3), that $\Psi$ satisfies (I4) and $\Theta$ satisfies (I5). Then 
	\begin{equation}
	\cW_{\Psi,\Theta}^{\f} := \{ \bx\in\bU  | \ \f(\bx)    \text{  is   }     (\Psi,\Theta)-\text{  approximable}\}\end{equation}
	has measure zero if $\sum_{\ba \in \Z^n\setminus\{0\}} \Psi(\ba) <\infty$.
	
\end{theorem}

 \subsection{Remarks}
 \begin{enumerate}               
\item We have assumed $S$ contains the infinite place in Theorem \ref{thm:main}. This is not a serious assumption, the proof in the case when $S$ contains only finite places needs some minor modifications but follows the same outline. In \cite{MoS1}, the (homogeneous) $S$-adic convergence case is proved in  slightly greater generality than in the present paper. Namely, instead of $\bbQ$, the quotient field of a finitely generated subring of $\bbQ$ is considered. This, more general formulation can also be incorporated into Theorem \ref{thm:main} without difficulty.   
\item The proofs of Theorems \ref{thm:divergence} and \ref{thm:divergence_infty}, follow the ubiquity framework used in \cite{BaBeVe} but need several new ideas to implement in the $p$-adic setting. 
\item Our proof for the convergence case, namely Theorem \ref{thm:main} blends techniques from the homogeneous results, namely \cite{KT, BKM, MoS1} and uses the transference principle developed by Beresnevich and Velani in the form used in \cite{BaBeVe}. The structure of the proof is the same as in \cite{BaBeVe}. 
\item We now undertake a brief discussion of the assumptions (I1) - (I5). The conditions (I1)-(I4) are assumed in \cite{MoS1} and, as explained in loc. cit., are assumed for convenience. Namely, as mentioned in \cite{MoS1}, the statement for any non-degenerate analytic manifold over $\bbQ_S$ follows from Theorem \ref{thm:main}. In \cite{BaBeVe}, the inhomogeneous parameter $\Theta$ is allowed to be $C^2$ when restricted to the nondegenerate manifold. However, we need to assume it to be analytic.
\item Theorem \ref{thm:divergence} is slightly more general than Theorem 1.2 of \cite{MoS2} in the homogeneous setting. In \cite{MoS2}, the approximating function is taken to be of the form
\begin{equation}
		\Psi(\ba)=\frac{1}{\|\ba\|^{n}}\psi(\|\ba\|), \ba\in\Z^{n+1} 
		\end{equation}
 which is a more restrictive class of approximating functions. For an $n$-tuple
$v = (v_1, \cdots, v_n)$ of positive numbers satisfying $v_1 + \cdots + v_n = n$, define the $v$-quasinorm $| ~ |_v$ on $\bbR^n$ by setting
$$ \|\bx\|_v := \max |x_i|^{1/v_i}. $$
Following \cite{BaBeVe} we say that a multivariable approximating function $\Psi$ satisfies property $\mathbf{P}$ if $\Psi(\ba) = \psi(\|\ba\|_v)$ for some approximating function $\psi$ and $v$ as above. As noted in loc. cit. when $v = (1, \dots, 1)$ we have that $\|\ba\|_v = \|\ba\|$ and any approximating function $\psi$ satisfies property $\mathbf{P}$, where $\psi$ is regarded as the function $\ba \to \psi(\|\ba\|)$. The proof of Theorem \ref{thm:divergence} can be modified to deal with the case of functions satisfying property $\mathbf{P}$.
\end{enumerate}

\subsection*{Acknowledgements} We would like to thank Yann Bugeaud for helpful discussions.
	
  \section{The divergence theorem: general setup}
 In this section we prove Theorem \ref{thm:divergence} using ubiquitous systems as in \cite{BaBeVe}. In \cite{BBKM}, the related notion of regular systems was used. 
 Suppose $S$ is a finite set of primes which may or may not contain infinity, with cardinality $l$. Suppose  $\bU=\prod_{\nu\in S}\bU_\nu$ is an open ball in $\prod_{i=1}^l\Q_{\nu_i}^{m_i}$, where if $\nu=\infty$ then $\Q_{\nu}=\R$. We assume $\f=(\f_\nu):\bU\to \Q_S^n$ is an analytic map satisfying condition (I2). For $\delta > 0$ and $Q > 1$ following \cite{BaBeVe} we define 
 $$\Phi^{\f}(Q,\delta) := \Big\{\bx \in \bU \text{ s.t }~\exists~\ba=(a_0,\ba_1)  \in \bbZ\times\bbZ^n\backslash \{0\}~|~\begin{array}{l} |a_0+ \ba_1 \cdot \f(\bx)|^l_S < \delta Q^{-{(n+1)}},\\  \Vert \ba\Vert\leq Q \end{array}\Big\},$$ when $S$ does not contain $\infty$ and otherwise define $$\Phi^{\f}(Q,\delta) := \Big\{\bx \in \bU \text{ s.t }~\exists~\ba \in \Z^n\backslash \{0\}, a_0\in\Z~|~\begin{array}{l} |a_0+ \ba \cdot \f(\bx)|^l_S < \delta Q^{-{n}},\\  \Vert \ba\Vert\leq Q \end{array}\Big\}.$$
 We now recall the definition of a $\mathit{nice}$ function.
 \begin{definition}[\cite{BaBeVe}, Definition 3.2]\label{nice}
 	We say that $\f$ is \textit{nice} at $\bx_0\in \bU$ if there exists a neighbourhood $\bU_0\subset \bU$ of $\bx_0$ and constants $0<\delta, w<1$ such that for any sufficiently small ball $\bB\subset \bU_0$ we have that 
 	\begin{equation}
 	\limsup_{Q\to \infty}|\Phi^{\f}(Q,\delta)\cap \bB |\leq w|\bB|.
 	\end{equation} 
 \end{definition}
 If $\f$ is \textit{nice} at almost every  $\bx_0$ in $\bU$ then $\f$ is called \textit{nice}.  \subsection{ Ubiquitous Systems in $\prod \Q_\nu^{m_\nu} $}
 Let us recall the definition of ubiquitous systems in $\prod \Q_\nu^{m_\nu} $ following \cite{BaBeVe}. Throughout, balls in $\Q_{\nu}^{m_\nu}$ and $\prod \Q_\nu^{m_\nu} $ are assumed to be defined in terms of supremum norms.  Let $\bU=(\bU_\nu)$ be a ball in $\prod \Q_\nu^{m_\nu} $ and $\mathcal{R}=(R_\alpha)_{\alpha\in J}$ be a family of subsets
 $R_\alpha\subset \prod \Q_\nu^{m_\nu} $ indexed by a countable set $J$. The sets $R_\alpha$ are referred to as \emph{resonant sets}. Throughout, $\rho\;:\;\R^+\to\R^+$
 will denote a function such that
 $\rho(r)\to0$ as $r\to\infty$.  Given a set $A\subset \bU$, let
 $$
 \Delta(A,r):=\{\bx\in \bU\;:\; \dist(\bx,A)<r\}
 $$
 where $\dist(\bx,A):=\inf\{|\bx-\ba|: \ba\in A\}$. Next, let
 $\beta\;:\;J\to \R^+\;:\;\alpha\mapsto\beta_\alpha$ be a positive
 function on $J$. Thus the function $\beta$ attaches a `weight'
 $\beta_\alpha$ to the set $R_\alpha$. We will assume that for every
 $t\in \N$ the set $J_t=\{\alpha\in J: \beta_\alpha\le 2^t\}$ is
 finite.\\
 \noindent\textbf{The intersection conditions:} There exists a constant $\gamma$ with $ 0 \leq \gamma \leq m=\sum m_\nu$ such that for any sufficiently large $t$ and for any
 $\alpha\in J_t$, $c\in\cR_\alpha$ and $0< \lambda \le \rho(2^t)$ the
 following conditions are satisfied:
 \begin{equation}\label{i1}
 \big|\bB(c, {\mbox{\small
 		$\frac{1}{2}$}}\rho(2^t))\cap\Delta(R_\alpha,\lambda)\big| \geq c_1 \,
 |\bB(c,\lambda)|\left(\frac{\rho(2^t)}{\lambda}\right)^{\gamma}
 \end{equation}
 \begin{equation}\label{i2}
 \big|\bB\cap
 \bB(c,3\rho(2^t))\cap\Delta(R_\alpha,3\lambda)\big| \leq c_2 \,
 |\bB(c,\lambda)| \left(\frac{r(\bB)}{\lambda}\right)^{\gamma} \
 \end{equation}
 where $\bB$ is an arbitrary ball centred on a resonant set with radius
 $r(\bB)\le 3 \, \rho(2^t)$. The   constants $c_1$ and $ c_2$  are
 positive and absolute. The constant $\gamma$ is referred to as the \emph{common dimension} of $\cR$.
 \begin{definition}
 	Suppose that there exists a ubiquitous function $\rho$ and an
 	absolute constant $k>0$ such that for any ball $\bB\subseteq \bU$
 	\begin{equation}\label{coveringproperty} 
 	\liminf_{t\to\infty} \left|\bigcup_{\alpha\in
 		J_t}\Delta(R_\alpha,\rho(2^t))\cap \bB\right| \ \ge \ k\,|\bB|.
 	\end{equation}
 	Furthermore, suppose that the intersection conditions \eqref{i1} and \eqref{i2}  are satisfied. Then
 	the system $(\mathcal{R}, \beta)$ is called \emph{locally ubiquitous in $\bU$
 		relative to $\rho$.}
 \end{definition}
 Let $(\mathcal{R},\beta)$ be a ubiquitous system in $\bU$ relative to $\rho$ and
 $\phi$ be an approximating function. Let $\Lambda(\phi)$ be the set
 of points $\bx\in \bU$ such that the inequality
 \begin{equation}\label{vb+}
 \dist(\bx,R_{\alpha})<\phi(\beta_\alpha)
 \end{equation}
 holds for infinitely many $\alpha\in J$.\\
 We are going to use this following ubiquity lemma from \cite{BaBeVe} in our main proof. 
 \begin{lemma}\label{ubi}
 	Let $\phi$ be an approximating function and $(\mathcal{R},\beta)$ be a locally ubiquitous system in $\bU$ relative to $\rho$. Suppose that there is a $0<\lambda<1$ such that $\rho(2^{t+1})<\lambda\rho({2^t})~\forall~ t \in \N.$ Then for any $s>\gamma,$ 
 	\begin{equation}
 	\mathcal{H}^s(\Lambda(\phi))=\mathcal{H}^s(\bU) \text{ if }\sum_{t=1}^\infty \frac{{\phi(2^t)}^{s-\gamma}}{{\rho(2^t)}^{m-\gamma}}=\infty.
 	\end{equation}
 \end{lemma}
\subsection{Some auxiliary lemmas}
 We will need the following strong approximation lemma mentioned in \cite{Zelo}. 
 \begin{lemma} \label{Strong} Suppose $(\xi_{\infty},\xi_{p_1},\cdots,\xi_{p_l})\in\R\times\Q_S$ where $S=\{p_1,\cdots,p_l\}$.
 	For any
 	$\bar \epsilon = (\epsilon_{\infty},\epsilon_{p_1},\cdots,\epsilon_{p_l})
 	\in \mathbb{R}_{>0}^{l+1}$ satisfying the inequality
 	\begin{equation}
 	\epsilon_{\infty} \geq
 	\frac{1}{2} \prod\epsilon_{p_i}^{-1} p_i,
 	\end{equation}
 	there exists a rational number
 	$r \in \mathbb{Q}$ such that
 	\begin{equation}
 	\begin{aligned}
 	&| r - \xi_{\infty} |_{\infty} \leq \epsilon_{\infty},
 	\\
 	&| r - \xi_{p_i} |_{p_i} \leq \epsilon_{p_i} \forall i=1,\cdots, l
 	,
 	\\
 	&| r |_{q} \leq 1
 	\quad \forall~q \notin S.
 	\end{aligned}
 	\end{equation}
 \end{lemma}

 Before we start proving  the main theorem in this section we would like to calculate a covolume formula of certain lattices. Also, we want to clarify some notations that are going to be used hereafter. Corresponding to each $\nu\in S$, the prime is $p_\nu$, and the corresponding norms are denoted by $\vert.\vert_\nu$ and $\Vert .\Vert_\nu$.
 \begin{lemma}\label{covolume} Suppose $S$ does not contain $\infty$ and  $\by=(\by^\nu)_{\nu}=(y_1^\nu,\cdots,y_n^\nu)\in\Q_S^n$ is such that $\vert y_i^\nu\vert_{\nu}\leq 1~\forall~\nu\in S$, $i=1,\cdots,n$ and $k_{\nu}\geq 0$. Then 
 	\begin{equation}\Gamma=\left\{
 	(z_0, z_1,\cdots, z_n)\in\Z^{n+1} :
 	\begin{array}{l}
 	\vert z_0 + z_1y_1^{\nu}+\cdots+ z_ny_n^{\nu}\vert_{\nu}\leq\frac{1}{p_{\nu}^{k_\nu}} ~\forall \nu,\\
 	|z_i|_\nu\leq \frac{1}{p_\nu} ~\forall \nu\in S\\
 	i=1,\cdots n
 	\end{array}\right\}
 	\end{equation} is a lattice in $\bbZ^{n+1} $
 	and $\Vol(\R^{n+1}/\Gamma)= \prod p_\nu^{k_{\nu}+n}$.
 \end{lemma}

 \begin{proof}
 	Note that $\Gamma$ is a discrete subgroup of $\bbZ^{n+1}$ and that $(\prod p_{\nu}^{k_\nu},0,\cdots,0)\in\Z^{n+1} $ is in $\Gamma$. Since the diagonal embedding of $\Z$ is dense in $\prod_{\nu} Z_{p_\nu}$, where $\Z_{p_\nu}$ is the unit disc in $\Q_{p_\nu},$ we can choose $q_i\in\Z$ such that 
 	\begin{equation}\label{q_conditions}
 	|q_i- \big(\prod p_\nu) y^\nu_i|_\nu\leq\frac{1}{p_\nu^{k_\nu}}, \forall \nu\in S.
 	\end{equation} 
 This implies that $(q_i,0,\cdots,-\prod p_{\nu},\cdots,0)\in \Gamma$ where $-\prod p_\nu$ is in the $(i+1)$th position. We claim that 
 $$\left\{(\prod p_{\nu}^{k_\nu},0,\cdots,0),(q_i,0,\cdots,-\prod p_\nu,\cdots,0)\ | \ i=1,\cdots,n \right\}$$ is a basis of $\Gamma$. The matrix comprising these elements as column vectors is as follows \[
 	A:=	\begin{bmatrix}
 	\prod p_{\nu}^{k_\nu} & q_1 &   & \dots &q_i &\dots & q_n\\
 	0    &  -\prod p_\nu &      &   \dots     &0   &\dots & 0  \\
 	\vdots & \vdots & &\vdots&\vdots &\vdots &\vdots,\\
 	0    &   0  &     &   \dots &-\prod p_\nu &\dots & 0\\
 	\vdots & \vdots  & &\vdots&\vdots&\vdots&\vdots \\
 	0    &     0 & &\dots &0  &\dots   & -\prod p_\nu
 	\end{bmatrix}.\]
 	We want to show that  if $\mathbf{z}=(z_0,z_1,\cdots,z_n)\in \Gamma,$ then there exists $\mathbf{s}=(s_o,s_1,\cdots,s_n)\in \Z^{n+1}$ such that $A\mathbf{s}=\mathbf{z}$. Note that \begin{equation}
 	A\inv \mathbf{z} = \left(\frac{(\prod p_{\nu})z_0+q_1z_1+\cdots+q_nz_n}{\prod p_\nu^{k_\nu+1}},-\frac{z_1}{\prod p_\nu},\cdots,-\frac{z_n}{\prod p_\nu}\right).
 	\end{equation}
 	As $\mathbf{z}\in \Gamma,$ we have that $p_\nu|z_i~\forall~ i=1,\cdots,n,$ hence $\prod  p_{\nu}|{z_i}$  for all $i$. Now it is enough to show that $p_\nu^{k_\nu+1} | (z_0p+q_1z_1+\cdots+z_nq_n)$ for all $\nu\in S$. Note that
 	$$z_0\prod p_\nu+z_1q_1+\cdots+z_nq_n= \prod p_\nu(z_0+z_1y^\nu_1+\cdots+z_ny^\nu_n)+z_1(q_1-(\prod p_\nu)y^\nu_1)+\cdots+z_n(q_n-(\prod p_\nu)y^\nu_n).$$
 	Now the conclusion follows since $\mathbf{z}\in\Gamma$ and using (\ref{q_conditions}).

 \end{proof}

 \section{Proof of the divergence theorem when $\infty \notin S$}
We begin by recalling  the following two Theorems from \cite{MoS2} which we will use to prove Theorem \ref{lemma:nice}. 
\begin{theorem}\label{<S}
	Let $S$ be as in (I0), $\bU$ be as in (I1), and assume that $\mathbf{f}$ satisfies (I2) and (I3). Then for any
	$\bx=(\bx_{\nu})_{\nu\in S}\in \bU$, one can find a neighborhood
	$\mathbf{V}=\prod V_{\nu}\subseteq \bU$ of $\bx$ and $\alpha>0$ with
	the following property: for any ball $\mathbf{B}\subseteq \mathbf{V}$,
	there exists $E>0$ such that for any choice of	$T_0,\cdots,T_n\ge 1$, $K_{\nu}>0$ and  $0<\delta^l\le \min\frac{1}{T_i}$,
 with $\delta^l T_0\cdots
			T_n\prod K_{\nu}\le 1$ one has
	\begin{equation}\label{<eqn}\left|\left\{\bx\in\mathbf{B}|\hspace{1mm}\exists\ 
	\ba=(a_0,\ba_1)\in\Z^{n+1}\setminus\{0\}:\begin{array}{l}\vert a_0+ \ba_1.\f(\bx)\vert_S<\delta\\
	\|\ba_1\nabla \f_{\nu}(\bx_\nu)\|<K_{\nu},\hspace{2mm}\\
	|a_i|_\infty<T_i, 1 \leq i \leq n\end{array}\right\}\right|\le
	E\hspace{.5mm}\varepsilon_1^{\alpha}|\mathbf{B}|,\hspace{5mm}
	\end{equation} 
	where
	$\varepsilon_1=\max\{\delta,(\delta^l{ T_0\cdots
			T_n\prod K_{\nu}})^{\frac{1}{(n+1)}}\},$ where $|S|=l$.
\end{theorem}\noindent
\begin{theorem}\label{>S}
	Assume that $\bU$ and $\f$ satisfy the conditions as in the previous theorem and $0<\varepsilon<\frac{1}{2l}$. For any $\delta>0$ and any ball $\bB\subset \bU$
		\begin{equation}\label{>eqn}\left|\left\{\bx\in\mathbf{B}|\hspace{1mm}\exists\ 
	\ba=(a_0,\ba_1)\in\Z^{n+1}\setminus\{0\}:\begin{array}{l}\vert a_0+ \ba_1 .\f(\bx)\vert^{l}_S<\delta T^{-(n+1)}\\
	\|\ba_1\nabla \f_{\nu}(\bx_\nu)\|>\Vert \ba\Vert_\infty^{-\varepsilon},\hspace{2mm}\forall\nu\in S\\
	|a_i|_\infty<T_i, 0 \leq i \leq n\end{array}\right\}\right|\le
	C\hspace{.5mm}\delta|\mathbf{B}|,\hspace{5mm}
	\end{equation} for some universal constant $C$.
\end{theorem}
The following Theorem shows that the functions in our set-up are nice. 
\begin{theorem}\label{lemma:nice}
	Assume that $\f:\bU\subset \Q_{p_1}^{m_1}\times\cdots\times \Q_{p_l}^{m_l}\to \Q_S^n$ satisfies condition (I2) (nondegenerate at $\bx\in \bU$). Then there exists a sufficiently small ball  $\bB_0\subset \bU$ centred at $\bx_0$ and a constant $C>0$ such that for any ball $\bB\subset \bB_0$ and any $1>\delta>0 $, for sufficiently large $Q$, one has 
	\begin{equation}
	|\Phi^{\f}(Q,\delta)\cap \bB |\leq C\delta |\bB|.
	\end{equation}
\end{theorem}
\begin{proof}
	Let $\Phi^{\f}(Q,\delta)\cap \bB\subset\Phi_{2}^{\f}(Q,\delta)\cup_{\nu\in S} \Phi_{1,\nu}^{\f}(Q,\delta)$ where, $$\Phi_{1,\nu}^{\f}(Q,\delta):=\left\{\bx\in\mathbf{B}|\hspace{1mm}\exists\ 
	\ba=(a_0,\ba_1)\in\Z^{n+1}\setminus\{0\}:\begin{array}{l}\vert a_0+ \ba_1 .\f(\bx)\vert^{l}_S<\delta Q^{-(n+1)}\\
	\|\ba_1\nabla \f_{\nu}(\bx_\nu)\|\leq\Vert \ba\Vert_\infty^{-\varepsilon} \hspace{2mm}\\
	|a_i|_\infty<Q, 0 \leq i \leq n\end{array}\right\}$$
	 and
$$\Phi_{2}^{\f}(Q,\delta):= \left\{\bx\in\mathbf{B}|\hspace{1mm}\exists\ 
\ba=(a_0,\ba_1)\in\Z^{n+1}\setminus\{0\}:\begin{array}{l}\vert a_0+ \ba_1 .\f(\bx)\vert^{l}_S<\delta Q^{-(n+1)}\\
\|\ba_1\nabla \f_{\nu}(\bx_\nu)\|>\Vert \ba\Vert_\infty^{-\varepsilon},\hspace{2mm}\forall \nu\in S\\
|a_i|_\infty<Q, 0 \leq i \leq n\end{array}\right\},$$ and $\varepsilon<\frac{1}{2l}$.
Now applying \ref{<S} we have $|\Phi_{1,\nu}^{\f}(Q,\delta)|\leq E ({\delta Q^{-\varepsilon}})^{\frac{\alpha}{n+1}}|\bB|$, where the constant $E$ is uniform. For large enough $Q$,  $E ({\delta Q^{-\varepsilon}})^{\frac{\alpha}{n+1}}\leq C.\delta $. By theorem \ref{>S} we have $|\Phi_{2}^{\f}(Q,\delta)|\leq C\delta|\bB|$, where $C$ is independent of the ball. Hence the theorem follows.
\end{proof}
We will now state the main two theorems of this section for the case where $S$ does not contain $\infty$. Let $\psi : \mathbb{N} \to \R_{+}$ be a decreasing function. 
 \begin{theorem} \label{thm:nice}
Assume that $\f:\bU\subset \Q_{p_1}^{m_1}\times\cdots\times \Q_{p_l}^{m_l}\to \Q_S^n$ is nice and satisfies the standing assumptions (I2) and (I3-ii) and  $s>\sum m_i-l=m-l$. Let $\Theta:\bU\to \Q_S$ be an analytic map satisfying assumption (I5). Let $\Psi(\ba)=\psi(\|\ba\|_\infty) ,\ba\in\Z^{n+1} $ be an approximating function. Then,
 \begin{equation}\label{main sum}
 	\mathcal{H}^s(\mathcal{W}^\f_{(\Psi,\Theta)}\cap\bU)=\mathcal{H}^s(\bU) \text{    if   } \sum (\Psi(\ba))^{\frac{s-m+l}{l}}=\infty,
 	\end{equation} where $m=\sum m_i$.
 	\end{theorem}
In view of Theorem \ref{lemma:nice}, Theorem \ref{thm:nice} implies Theorem \ref{thm:divergence}. Note that condition (I3) implies the nondegeneracy of $\f$ at every point of $\bU$.

  Now we will construct a ubiquitous system which will give  the main result of this section.
 \begin{theorem} \label{ubiquity}
 	Let $\bx_0\in \bU=\prod \bU_\nu$ be such that $\f:\bU\subset \prod \Q_{\nu}^{m_\nu}\to \Q_S^n$ is \textit{nice} at $\bx_0$ and satisfies (I3-ii). Then there is a neighbourhood $\bU_0$ of $\bx_0,$ constants $\kappa_0>0$ and $\kappa_1>0$ and a collection $\cR:=(R_F)_{F\in\mathcal{F}_n}$ of sets $R_F\subset \widetilde{R_F}\cap \bU_0$ such that the system $(\cR,\beta)$ is locally ubiquitous in $\bU_0$ relative to $\rho(r)=\kappa_1r^{-\frac{n+1}{l}} $ with common dimension $\gamma:=(\sum m_{\nu})-l,$ where
 	 $$
 	 \mathcal{F}_n:=\left\{F:\bU\to\Q_S  \left |\begin{array}{l} F(\bx)= a_0+a_1f_1(\bx)+\cdots+a_nf_n(\bx),\\
 	 \ba=(a_0,a_1,\cdots,a_n)\in\Z^{n+1}\setminus\{\mathbf{0}\} \end{array}\right. \right \}	$$ 
 	 and given $F\in\mathcal{F}_n$ we consider semi resonant sets
 	 \begin{equation}
 	 \widetilde{R_F}:=\{\bx\in\bU:\ (F+\Theta)(\bx) \ =\ 0\}
 	 \end{equation}
 	 and $$
 	 \beta:\ \mathcal{F}_n\to \R^+\ : F\to \ \beta_F=\kappa_0|(a_0,a_1,\cdots,a_n)|=\kappa_0\Vert\ba\Vert_\infty.
 	 $$
 \end{theorem}
 \begin{proof}
  Let $\pi_\nu:\Q_\nu^{m_\nu}\to\Q_{\nu}^{m_\nu-1}$ be the projection map given by $$\pi_\nu(x^\nu_1,x^\nu_2,\cdots,x^\nu_m)=(x^\nu_2,\cdots,x^\nu_m),$$ and $\pi=(\pi_\nu):\prod\Q_\nu^{m_\nu} \to \prod\Q_\nu^{m_\nu-1}$
   and let 
	\begin{equation}
 \widetilde\bV:=\pi(\widetilde R_F\cap\bU_0),
 \\
  \bV=\bigcup_{3\rho(\beta_F)-\text{balls} B\subset \widetilde{\bV}}\frac{1}{2}B
 	\end{equation}
 and
 \begin{equation}
 R_F=\left\{\begin{array}{l}
 
\pi\inv(\bV)\cap\widetilde{R_F}  \ \text{if} \ \ |\partial_1(F^\nu+\Theta^\nu)(\bx_\nu)|_\nu> \lambda_\nu|\nabla(F^{\nu}+\Theta^\nu)(\bx_\nu)|_\nu \ \forall \nu\in S, \ \forall \ \bx=(\bx_\nu)\in \bU_0\\

\emptyset  \ \ \ \ \ \ \ \ \ \ \ \ \ \ \ \  \text{otherwise}.
 \end{array}\right 
.\end{equation} 
where $0<\lambda_\nu<1$ is fixed for each $\nu\in S$. \\

We claim that the $R_F$ are resonant sets. The intersection property, namely (\ref{i1}) and (\ref{i2}), when $S$ is singleton, can be checked exactly as in the case of real numbers as accomplished in \cite{BaBeVe}, Proposition 5. We only need to note that implicit function theorem for $C^l(U)$ in $\R^n$ was used in \cite{BaBeVe}. The implicit function theorem in $\Q_p$ holds for analytic maps and all our maps have been assumed analytic, so the proof in \cite{BaBeVe} goes through verbatim. \\
\noindent
We now consider the case when $S$ contains more than one prime. One can easily see that $\widetilde {R_F}=\prod\widetilde {R_F^\nu}$, where $\widetilde{R_F^\nu}$ are the semi resonant sets coming from single $p_\nu$ case. Again it can be easily seen that $R_F=\prod R_F^\nu$, where $R_F^\nu$ is the resonant set corresponding to $\{p_\nu\}$. Note that $\Delta(\prod A_\nu,r)=\prod \Delta(A_\nu,r)$, where $A_\nu\subset \Q_\nu^{m_\nu}$. Since the intersection and product operations on sets can interchanged and for singleton $S$ ubiquitous intersection properties hold, we get intersection properties to be satisfied for any $S$ with common dimension being $\sum(m_\nu-1)$ after using the property of product measure.

It remains to check the covering property (\ref{coveringproperty}) to establish ubiquity. Without loss of generality we will assume that the ball $\bU_0$ in the definition of (\ref{nice}) satisfies
\begin{equation}
\diam{\bU_0}\leq \min_\nu \frac{1}{p_\nu}.
\end{equation}
From the Definition \ref{nice} of  $\f$ being nice at $\bx_0,$ there exist fixed $0<\delta,w<1$ such that for any arbitrary ball $\bB\subset\bU_0,$ 
\begin{equation}
	\limsup_{Q\to \infty}|\Phi^{\f}(Q,\delta)\cap \frac{1}{2}\bB |\leq w|\frac{1}{2}\bB|.
\end{equation}
 So for sufficiently large $Q$ we have that 
 $$
 |\frac{1}{2}\bB\setminus \Phi^{\f}(Q,\delta)|\geq \frac{1}{2}(1-w)|\frac{1}{2}\bB|=2^{-(\sum (m_\nu)+1)}(1-w)|\bB|. 
 $$
 Therefore it is enough to show that 
 \begin{equation}
 \frac{1}{2}\bB\setminus \Phi^{\f}(Q,\delta)\subset\bigcup _{F\in\mathcal{F}_n\\
 \beta_F\leq Q}\Delta(R_F,\rho(Q))\cap\bB.
 \end{equation}
 Suppose $\bx\in \frac{1}{2}\bB\setminus \Phi^{\f}(Q,\delta).$ Consider the lattice
 \begin{equation}
 \Gamma_{\bx}=\left\{(a_0,a_1,\cdots,a_n)\in\Z^{n+1}: \begin{array}{l}|a_0+a_1f_1(\bx)+\cdots+a_nf_n(\bx)|^l_S<\delta Q^{-(n+1)}\\
 |a_i|_\nu\leq\frac{1}{p_\nu} \ \forall \nu\in S, \ \forall \ {1\leq i\leq n}\end{array}\right\},
 \end{equation}
 and the convex set  $K=[-Q,Q]^{n+1}$ in $\R^{n+1}$. Note that $$
 |a_o+a_1f_1(\bx)+\cdots+a_nf_n(\bx)|_\nu<\delta^{\frac{1}{l}} Q^{-\frac{n+1}{l}}$$
if and only if
$$|a_o+a_1f_1(\bx)+\cdots+a_nf_n(\bx)|_\nu\leq {p_\nu^{[\log_{p_\nu}(\delta Q^{-(n+1)})^\frac{1}{l}]}}.
 $$
 So by Lemma \ref{covolume} we have that 
 \begin{eqnarray*}
 \Vol(\R^{n+1}/\Gamma_{\bx}) & = & \prod p_\nu^n.p_\nu^{-[\log_{p_\nu}(\delta Q^{-(n+1)})^\frac{1}{l}]} \\
    & \leq & \prod p_\nu^{n}\frac{1}{\prod p_\nu^{log_{p_\nu}{(\delta Q^{-(n+1)})^\frac{1}{l}-1}}}\\
    & \leq &  \prod Q^{\frac{n+1}{l}}\frac{p_\nu^{n+1}}{\delta^\frac{1}{l}}\\
    &=& Q^{n+1}.\frac{\prod p_\nu^{n+1}}{\delta}. 
  \end{eqnarray*}
Using the fact that $\bx\notin \Phi^{\f}(Q,\delta) $ we get the first minima $\lambda_1=\lambda_1(\Gamma_{\bx},K)>1$. Therefore using Minkowski's theorem on successive minima, we have that $$
2^{n+1}Q^{n+1}\lambda_1.\lambda_2.\cdots.\lambda_{n+1}\leq 2^{n+1}\Vol(\R^{n+1}/\Gamma_{\bx})\leq 2^{n+1}Q^{n+1}\frac{\prod p_\nu^{n+1}}{\delta}.
$$
This implies that $\lambda_{n+1}\leq \frac{\prod p_\nu^{n+1}}{\delta}.$ By the definition of $\lambda_{n+1}$ we get $n+1$ linearly independent integer vectors $\ba_j=(a_{j,0},\cdots,a_{j,n})\in\Z^{n+1}(0\leq j\leq n)$ such that the functions $F_j$ given by $$
 F_j(\by)=a_{j,0}+a_{j,1}f_1(\by)+\cdots+a_{j,n}f_n(\by)
$$ satisfy 
\begin{equation}\label{conditions}
\left\{ \begin{array}{l}
|F_j(\bx)|^l_S<\delta Q^{-(n+1)}\\
|a_{j,i}|_\infty\leq Q.\frac{\prod p_\nu^{n+1}}{\delta}\\
|a_{j,i}|_\nu\leq\frac{1}{p_\nu} \text{ for } 0\leq j \leq n, 1\leq i\leq n\ \forall \nu\in S.
\end{array}\right.
\end{equation} 
As $\lambda_1>1$ so for every $0\leq j \leq n$ there exists at least one $0\leq j^\star\leq n$ such that $|a_{j,j^\star}|_\infty>Q$.

Now consider the following system of linear equations,
\begin{equation}\label{linear}
\begin{array}{l}
\eta_0F_0(\bx)+\eta_1F_1(\bx)+\cdots+\eta_nF_n(\bx)+\Theta(\bx)=0\\
\eta_0\partial_1F_0(\bx)+\eta_1\partial_1F_1(\bx)+\cdots+\eta_n\partial_1F_n(\bx)+\partial_1\Theta(\bx)=1\\
\eta_0a_{0,j}+\cdots+\eta_na_{n,j}=0 \ \  (2\leq j \leq n).
\end{array}
\end{equation}
Since $f_1(\bx)=x_1 $, the determinant of this aforementioned  system is $\det(a_{j,i})\neq 0$. Therefore there exists a unique solution to the system, say $(\eta_0,\eta_1,\cdots,\eta_n)\in \Q_S^n$. By the argument above, there is at least one $|a_{j,i} |_\infty > Q$. Without loss of generality assume $|a_{0,0}|_\infty >Q$. Using the strong approximation Theorem \ref{Strong} for $i=0,\cdots,n$ we get $r_i\in\Q$ such that
\begin{equation}\label{r_i}
\begin{aligned}
& |r_i-2\prod p_\nu|_\infty\leq\prod p_\nu \text{ if }  a_{i,0}>0 \text{   otherwise  } |r_i+2\prod p_\nu|_\infty<\prod p_\nu,\\
&
|r_i-\eta^\nu_i|_\nu\leq 1 \ \forall~\nu\in S,\\
& 
|r_i|_{\nu}\leq 1
\quad \text{for all} \ \text{ primes }\nu\notin S.
\end{aligned}
\end{equation}
Now take the function 
\begin{equation}\begin{aligned}
F(\by)=r_0F_0(\by)+r_1F_1(\by)+\cdots+r_nF_n(\by)\\
=a_0+a_1f_1(\by)+\cdots+a_nf_n(\by),
\end{aligned}
\end{equation} where 
\begin{equation}\label{a_i}
a_i=r_0a_{0,i}+r_1a_{1,i}+\cdots+r_na_{n,i},
\  \forall \ i=0,\cdots,n.
\end{equation}

We claim that\\

\textbf{Claim $1$.}The $a_i$s are all integers.\\ 

From (\ref{r_i}) and (\ref{a_i}) we get 
\begin{equation}\label{claim1.1}
|a_i|_\nu\leq 1, \ \forall \ \ {i}=0,\cdots,n \text{ for } \nu\notin S
\end{equation}
and by (\ref{r_i}), (\ref{linear}) and (\ref{conditions}) we have 
\begin{equation}\begin{aligned}
|a_i|_\nu\leq \max_{j=0,\cdots,n} \{|\eta^\nu_j-r_j|_\nu|a_{j,i}|_\nu\}\leq 1 \ \forall \nu\in S
\quad \text{and for } i=2,\cdots,n.
\end{aligned}
\end{equation} So $a_i's$ are all integers for $i=2,\cdots,n$.
Now note that $$
F(\bx)+\Theta(\bx)\\
=(r_0-\eta_0)F_0(\bx)+\cdots+(r_n-\eta_n)F_n(\bx).
$$ Let us denote $F=(F^\nu)$ and $\Theta=(\Theta^\nu) $ and we are considering $\Q\subset \Q_S$, diagonally embedded. Therefore we have \begin{equation}\label{condition1}
|(F+\theta)(\bx)|^l_S\leq (\delta Q^{-(n+1)}).
\end{equation}
Again $$
\partial_1(F+\Theta)(\bx)=(r_0-\eta_0)\partial_1F_0(\bx)+\cdots+(r_n-\eta_n)\partial_1F_n(\bx)+1.
$$ Since $\forall \nu\in S$, $1\leq i\leq n$ and $0\leq j\leq n$, $|a_{j,i}|_\nu\leq\frac{1}{p_\nu}$ so $|\partial_1F^\nu_j(\bx)|_\nu\leq \frac{1}{p_\nu}$ and thus by (\ref{r_i}) we get 
\begin{equation}\label{partial_condition}
1-\frac{1}{p_\nu}\leq|\partial_1(F^\nu+\Theta^\nu)(\bx)|_\nu\leq 1.
\end{equation} 
Now we can show that $a_1$ and $a_0$ are also integers. Since $f_1(\by)=y_1,$ we have 
\begin{equation}
a_1=\partial_1(F+\Theta)(\bx)-\partial_1\Theta(\bx)-\sum_{j =2}^{n}a_j\partial_1f_j(\bx)
\end{equation}
which implies that $|a_1|_\nu\leq 1 \ \forall~\nu\in S$. This together with (\ref{claim1.1}) proves that  $a_1$ is an integer. We similarly prove that $a_0$ is an integer. We can write
\begin{equation}\begin{aligned}\label{a_0}
a_0=(F+\Theta)(\bx)-\Theta(\bx)-\sum_{j =1}^{n}a_jf_j(\bx).
\end{aligned}
\end{equation}
This implies that $|a_0|_\nu\leq 1\ \forall~\nu\in S$ and thus by (\ref{a_0}) and (\ref{claim1.1}) we get that $a_0$ is integer. So the first claim is proved.

Now we look at the infinity norm of the integers $a_i$.
By (\ref{a_i}), (\ref{conditions}) and (\ref{r_i}) we have 
\begin{equation}\label{a_infty}
\begin{aligned}
|a_i|_\infty\leq|r_0a_{0,i}+\cdots+r_na_{n,i}|_\infty\\
\leq 3(\prod p_\nu)(n+1)Q.\frac{\prod p_\nu^{n+1}}{\delta}
\end{aligned}
\quad \text{ for } i=0,1,\cdots,n.
\end{equation}
 By the choice of $r_i$ we have  $a_0>0$ and using the fact that $Q<|a_{0,0}|_\infty$ we get that $|a_0|_\infty>\prod p_\nu Q$ and therefore $\Vert\ba\Vert_\infty>\prod p_\nu Q$.

So by (\ref{a_infty}) and the previous observation we get 
\begin{equation}\label{beta}
\frac{1}{3(\prod p_\nu)(n+1)}\prod p_\nu^{-(n)}\delta.Q<\beta_F=\frac{1}{3(\prod p_\nu)(n+1)}\prod p_\nu^{-(n+1)}\delta||\ba||_\infty\leq Q,\end{equation} 
here $\kappa_0=\frac{1}{3(\prod p_\nu)(n)}\prod p_\nu^{-(n)}\delta$. 
Note that  for all $\by\in\bU_0$ we have 
\begin{equation}
\partial_1(F+\Theta)(\bx)=\partial_1(F+\Theta)(\by)+\sum_{j=1}^m\Phi_{j1}(\partial_1(F+\Theta))(\star)(x_j-y_j)
\end{equation}
 where $\star$ is from the coefficients of $\bx$ and $\by$. By using (\ref{partial_condition}) and by the fact that $\diam(\bU_0)\leq\min \frac{1}{p_\nu}$ we have
 \begin{equation}
 |\partial_1(F^\nu+\Theta^\nu)(\by_\nu)|_\nu\geq 1-\frac{2}{p_\nu} \ \ \forall \  (\by_\nu)\in\bU_0 \ \forall~\nu\in S.
 \end{equation}
 So $F$ satisfies $|\partial_1(F^\nu+\Theta^\nu)(\bx_\nu)|_\nu> (1-\frac{2}{p_\nu})|\nabla(F^\nu+\Theta^\nu)(\bx_\nu)|_\nu \ \forall \ \bx=(\bx_\nu)\in \bU_0$ and thus by the constructions $\Delta(R_F,\rho(Q))\neq \emptyset$.\\ 

 \textbf{Claim $2$.} $\bx \in \Delta(R_F,\rho(Q))$.\\
 
 We set $r_0 := \diam(\bB)$ and define the function 
 $$g^\nu(\xi^\nu) :=(F^\nu+\Theta^\nu)(x^\nu_1+\xi^\nu,x_2,\cdots,x^\nu_{m_\nu}), \text { where } |\xi|_S=|(\xi^\nu)|_S<r_0$$ and for all $\nu\in S$.
 Then \begin{equation}
 \begin{aligned}
 |g^\nu(0)|_p=|(F^\nu+\Theta^\nu)(\bx)|^l_\nu<\delta Q^{-(n+1)} \\
 \text{ and } |{g^\nu}'(0)|_p=|\partial_1(F^\nu+\Theta^\nu)(\bx)|_\nu>1-\frac{1}{p_\nu}.
\end{aligned}
\end{equation} 
Now applying Newton's method there exists $\xi_0=(\xi^\nu_0)$ such that $g^\nu(\xi^\nu_0)=0$ and $|\xi^\nu_0|_\nu<\frac{p_\nu}{(p_\nu-1)}(\delta Q^{-(n+1)})^{\frac{1}{l}}$. For sufficiently large $Q$ we get $\bx_{\xi_0}=(x_1+\xi_0,x_1,\cdots,x_n)\in \bB,$ that $(F+\Theta)(\bx_{\xi_0})=0$ and that $\Vert\bx-\bx_{\xi_0}\Vert_S\leq(\max \frac{p_\nu}{(p_\nu-1)})(\delta Q^{-(n+1)})^{\frac{1}{l}}$. Then we will argue exactly same as in \cite{BaBeVe}. We recall the argument for the sake of completeness. By  the Mean Value Theorem we will get 
$$\begin{aligned}
|(F+\Theta)(\by)|^l_S \ll Q^{-(n+1)}\\
	\text{ for any } \Vert\by-\bx_{\xi_0}\Vert_S \ll Q^{-\frac{n+1}{l}}.
\end{aligned}
 $$  
 Then by (\ref{beta})  and using the same argument as above tells us that for sufficiently large $Q>0$ the ball of radius $\rho(\beta_F)$ centred at $\pi\bx_{\xi_0}$ is contained in $\widetilde{\bV}$. This ultimately gives $\bx_{\xi_0}\in R_F$ . Since 
 $$\Vert\bx-\bx_{\xi_0}\Vert_S\leq (\max \frac{p_\nu}{(p_\nu-1)})(\delta Q^{-(n+1)})^{\frac{1}{l}}$$ 
 so $\bx\in\Delta(R_F,\rho(Q))$ where $\rho(Q)= (\max \frac{p_\nu}{(p_\nu-1)})(\delta Q^{-(n+1)})^{\frac{1}{l}}=\kappa_1Q^{-\frac{n+1}{l}}$.
 Therefore $\bx\in \Delta(R_F,\rho(Q))$ for some $F\in\mathcal{F}_n $ such that $\beta_F\leq Q$ and this completes the proof of the Theorem. 
 \end{proof}
\subsubsection{ Proof of Theorem \ref{thm:nice}}
Now using Theorem \ref{ubiquity} and Lemma \ref{ubi} we can complete the proof of Theorem \ref{thm:nice}.

Fix $\bx_0\in \bU$ and let $\bU_0$ be the neighbourhood of $\bx_0$ which comes from (\ref{ubiquity}). We need to show that 
$$\mathcal{H}^s(\mathcal{W}^\f_{(\Psi,\Theta)}\cap\bU_0)=\mathcal{H}^s(\bU_0)
$$ if the series in (\ref{main sum}) diverges. Consider $\phi(r):=\psi(\kappa_0\inv r)^\frac{1}{l} $. Our first aim is to show that
\begin{equation}
\Lambda(\phi)\subset \mathcal{W}^\f_{(\Psi,\Theta)}.
\end{equation}
Note that $\bx\in \Lambda(\phi)$ implies the existence of infinitely many  $F\in\mathcal{F}_n $ 
such that $\dist(\bx,R_F)<\phi(\beta_F)$. For such $F\in\mathcal{F}_n$ there exists  $\bz\in\bU_0$ such that $(F+\Theta)(\bz)=0$ and $\Vert\bx-\bz\Vert_S<\phi(\beta_F)$. By the mean value theorem,
$$
(F+\Theta)(\bx)=(F+\Theta)(\bz)+ \nabla(F + \Theta)(\bx)\cdot (\bx - \bz) + \sum_{i,j}\Phi_{ij}(F+\Theta)(\star)(x_i - z_i)(x_j-z_j), $$
 where $\star$ comes from the coefficients of $\bx$ and $\bz $. Then we have that 
 \begin{equation}
 |(F+\Theta)(\bx)|_S^l\leq\Vert\bx-\bz\Vert^l_S<\phi(\beta_F)^l=\phi(\kappa_0 \Vert\ba\Vert_\infty)^l=\Psi(\ba).
 \end{equation}
Hence $\Lambda(\phi)\subset \mathcal{W}^\f_{(\Psi,\Theta)} $. Now the Theorem will follow if we can show that
$$\sum_{t=1}^{\infty} \frac{\phi(2^t)^{s-m+l}}{\rho(2^t)^l}=\infty ,$$ where $m=\sum m_\nu$. 
Observe that
$$\sum_{t=1}^{\infty} \frac{\phi(2^t)^{s-m+l}}{\rho(2^t)^l}\asymp \sum_{t=1}^\infty (\psi(\kappa_0\inv 2^t))^{\frac{s-m+l}{l}}\frac{1}{(\rho(2^t))^l}\\
\asymp \sum_{t=1}^\infty (\psi(\kappa_0\inv 2^t))^{\frac{s-m+l}{l}}2^{t(n+1)}$$

As $\psi$ is an approximating function so we got that the above series $$\gg\sum_{t=1}^\infty \sum_{\kappa_0\inv 2^t<\Vert\ba\Vert\leq\kappa_0\inv2^{t+1} }(\psi(\Vert\ba\Vert))^{\frac{s-m+l}{l}}\asymp  \sum_{\ba\in\Z^{n+1}\setminus{0}}(\psi(\Vert\ba\Vert))^{\frac{s-m+l}{l}}\\
$$ $$=\sum_{\ba\in\Z^{n+1}\setminus{0}}\Psi(\ba)^{\frac{s-m+l}{l}}=\infty.
$$
This completes the proof of the Theorem.

\section{Proof of the divergence theorem when $\infty \in S$}
We will state the main theorem of this section for the case where $S$  contains $\infty$. Let $\psi : \mathbb{N} \to \R_{+}$ be a decreasing function. 
\begin{theorem} \label{thm:nice_R}
	Assume $S$ contains $\infty$ and $\f:\bU\subset \prod_{\nu\in S}\Q_{\nu}^{m_\nu}\to \Q_S^n$ is nice and satisfies the standing assumptions (I2) and (I3-ii) and  $s>\sum m_i-l=m-l$. Let $\Theta:\bU\to \Q_S$ be an analytic map satisfying assumption (I5). Let $\Psi(\ba)=\psi(\|\ba\|_\infty) ,\ba\in\Z^{n} $ be an approximating function. Then,
	\begin{equation}\label{main sum_R}
	\mathcal{H}^s(\widetilde{\mathcal{W}^\f_{(\Psi,\Theta)}}\cap\bU)=\mathcal{H}^s(\bU) \text{    if   } \sum (\Psi(\ba))^{\frac{s-m+l}{l}}=\infty,
	\end{equation} where $m=\sum m_\nu$.
\end{theorem}
\noindent 
 We recall Theorems $1.3$ and $1.2$ from \cite{MoS1}.
\begin{theorem}\label{<}
	Let $S$ be as in (I0), $\bU$ be as in (I1), and assume that $\mathbf{f}$ satisfies (I2) and (I3). Then for any
	$\bx=(\bx_{\nu})_{\nu\in S}\in \bU$, one can find a neighborhood
	$\mathbf{V}=\prod V_{\nu}\subseteq \bU$ of $\bx$ and $\alpha_1 >0$ with
	the following property: for any ball $\mathbf{B}\subseteq \mathbf{V}$,
	there exists $E>0$ such that for any choice of $0<\delta\le 1$,
	$T_1,\cdots,T_n\ge 1$, and $K_{\nu}>0$ with $\delta{ (\frac{T_1\cdots
			T_n}{\max_i T_i})}\prod K_{\nu}\le 1$ one has
	\begin{equation}\label{<eqn}\left|\left\{\bx\in\mathbf{B}|\hspace{1mm}\exists\ 
	\ba\in\Z^n\setminus\{0\}:\begin{array}{l}|\langle \ba .\f(\bx) \rangle|^{l}_S<\delta\\
	\|\ba\nabla \f_{\nu}(\bx_\nu)\|<K_{\nu},\hspace{2mm}\nu\in S\\
	|a_i|<T_i, 1 \leq i \leq n\end{array}\right\}\right|\le
	E\hspace{.5mm}\varepsilon_1^{\alpha_1}|\mathbf{B}|,\hspace{5mm}
	\end{equation} 
	where
	$\varepsilon_1=\max\{\delta^\frac{1}{l},(\delta{ (\frac{T_1\cdots
			T_n}{\max_i T_i})}\prod K_{\nu})^{\frac{1}{\l(n+1)}}\},$ where $|S|=l$.
\end{theorem}\noindent

\begin{theorem}\label{>S_R}
	Let's assume $\bU$ and $\f$ satisfy the conditions as in the previous theorem and $0<\varepsilon<\frac{1}{4nl^2}$. For any $\delta>0$ and any ball $\bB\subset \bU$
	\begin{equation}\label{>eqn_R}
	\left|\left\{\bx\in\mathbf{B}|\hspace{1mm}\exists\ 
	\ba=(a_0,\ba_1)\in\Z^{n+1}\setminus\{0\}:\begin{array}{l}\vert a_0+ \ba_1 .\f(\bx)\vert^{l}_S<\delta \prod T_i^{-1}\\
	\|\ba_1\nabla \f_{\nu}(\bx_\nu)\|>\Vert \ba\Vert_\infty^{-\varepsilon},\hspace{2mm}\forall\nu\in S\setminus\infty\\	\|\ba_1\nabla \f_{\infty}(\bx_\infty)\|>\Vert \ba\Vert_\infty^{1-\varepsilon}\\
	\frac{T_i}{2}\leq |a_i|_\infty<T_i, 1 \leq i \leq n\end{array}\right\}\right|\le
	C\hspace{.5mm}\delta|\mathbf{B}|,\hspace{5mm}
	\end{equation} for some universal constant $C$ and large enough $\max{T_i}$.
\end{theorem}
The following Theorem shows that the functions in our set-up are nice. 
\begin{theorem} \label{lemma:nice_infy}
	Assume that $\f:\bU\subset \prod_\nu\Q_{\nu}^{m_\nu}\to \Q_S^n$ satisfies condition (I2) (nondegenerate at $\bx\in \bU$). Then there exists a sufficiently small ball  $\bB_0\subset \bU$ centred at $\bx_0$ and a constant $C>0$ such that for any ball $\bB\subset \bB_0$ and any $1>\delta>0 $, for sufficiently large $Q$, one has 
	\begin{equation}
	|\Phi^{\f}(Q,\delta)\cap \bB |\leq C\delta |\bB|.
	\end{equation}
\end{theorem}
\begin{proof}
	We take $\varepsilon<\frac{1}{4nl^2}$.
	Note that  $\Phi^{\f}(Q,\delta)\cap \bB\subset\Phi_{2}^{\f}(Q,\delta)\cup_{\nu\in S} \Phi_{1,\nu}^{\f}(Q,\delta)$ where, $$\Phi_{1,\nu}^{\f}(Q,\delta):=\left\{\bx\in\mathbf{B}|\hspace{1mm}\exists\ 
	\ba=(a_0,\ba_1)\in\Z^{n+1}\setminus\{0\}:\begin{array}{l}\vert a_0+ \ba_1 .\f(\bx)\vert^{l}_S<\delta Q^{-n}\\
	\|\ba_1\nabla \f_{\nu}(\bx_\nu)\|\leq\Vert \ba\Vert_\infty^{-\varepsilon+k_\nu}\hspace{2mm}\\
	|a_i|_\infty<Q, 0 \leq i \leq n\end{array}\right\},$$
	$k_\nu=0  \text{ if } \nu\in S\setminus\infty, k_\infty=1$
	and
	$$\Phi_{2}^{\f}(Q,\delta):= \left\{\bx\in\mathbf{B}|\hspace{1mm}\exists\ 
	\ba=(a_0,\ba_1)\in\Z^{n+1}\setminus\{0\}:\begin{array}{l}\vert a_0+ \ba_1 .\f(\bx)\vert^{l}_S<\delta Q^{-n}\\
	\|\ba_1\nabla \f_{\nu}(\bx_\nu)\|>\Vert \ba\Vert_\infty^{-\varepsilon},\hspace{2mm}\forall \nu\in S\setminus\infty\\ \|\ba_1\nabla \f_{\infty}(\bx_\infty)\|>\Vert \ba\Vert_\infty^{1-\varepsilon}\\
	|a_i|_\infty<Q, 0 \leq i \leq n\end{array}\right\}.$$
	Applying Theorem \ref{<} with the choice of $K_\nu=\Vert \ba_1\Vert_\infty^{-\varepsilon} ,K_{\nu'}=1 \text{ if } \nu'\neq\nu, K_\infty=\Vert \ba_1\Vert_\infty$ we have that 
	$$|\Phi_{1,\nu}^{\f}(Q,\delta)|\leq E ({\delta Q^{-\varepsilon}})^{\frac{\alpha}{(n+1)l}}|\bB|$$ for $\nu\in S\setminus\infty$, with uniform $E$. In order to compute $|\Phi_{1,\infty}^{\f}(Q,\delta)|$ we apply Theorem \ref{<}  taking $K_\infty=\Vert \ba_1\Vert_\infty^{1-\varepsilon}, K_\nu=1 ~\forall~\nu\in S\setminus\infty$ to get  $$|\Phi_{1,\infty}^{\f}(Q,\delta)|\leq E ({\delta Q^{-\varepsilon}})^{\frac{\alpha}{(n+1)l}}|\bB|.$$ For large enough $Q$,  $E ({\delta Q^{-\varepsilon}})^{\frac{\alpha}{(n+1)l}}\leq C.\delta $. By Theorem \ref{>S_R} we have $$|\Phi_{2}^{\f}(Q,\delta)|\leq C\delta(\sum_{k\geq 0} \frac{1}{2^{kn}} )|\bB|,$$ where $C$ is independent of the ball. Hence the result follows.
\end{proof}
In view of Theorem \ref{lemma:nice_infy}, Theorem \ref{thm:nice_R} implies Theorem \ref{thm:divergence_infty}. Note that condition (I3) implies the nondegeneracy of $\f$ at every point of $\bU$.

\begin{theorem}\label{ubiquity_R}
	Let $\bx_0\in \bU=\prod \bU_\nu$ be such that $\f:\bU\subset \prod \Q_{\nu}^{m_\nu}\to \Q_S^n$ is \textit{nice} at $\bx_0$ and satisfies (I3-ii). Then there is a neighbourhood $\bU_0$ of $\bx_0,$ constants $\kappa_0>0$ and $\kappa_1>0$ and a collection $\cR:=(R_F)_{F\in\mathcal{F}_n}$ of sets $R_F\subset \widetilde{R_F}\cap \bU_0$ such that the system $(\cR,\beta)$ is locally ubiquitous in $\bU_0$ relative to $\rho(r)=\kappa_1r^{-\frac{n}{l}} $ with common dimension $\gamma:=(\sum m_{\nu})-l,$ where
	$$
	\mathcal{F}_n:=\left\{F:\bU\to\Q_S  \left |\begin{array}{l} F(\bx)= a_0+a_1f_1(\bx)+\cdots+a_nf_n(\bx),\\
\tilde\ba=(a_0,\ba)= (a_0,a_1,\cdots,a_n)\in\Z\times\Z^{n}\setminus\{\mathbf{0}\} \end{array}\right. \right \}	$$ 
	and given $F\in\mathcal{F}_n$ we consider semi resonant sets
	\begin{equation}
	\widetilde{R_F}:=\{\bx\in\bU:\ (F+\Theta)(\bx) \ =\ 0\}
	\end{equation}
	and $$
	\beta:\ \mathcal{F}_n\to \R^+\ : F\to \ \beta_F=\kappa_0|(a_1,\cdots,a_n)|=\kappa_0\Vert\ba\Vert_\infty.
	$$
\end{theorem}
\begin{proof}
	Let $\pi_\nu:\Q_\nu^{m_\nu}\to\Q_{\nu}^{m_\nu-1}$ be the projection map given by $$\pi_\nu(x^\nu_1,x^\nu_2,\cdots,x^\nu_m)=(x^\nu_2,\cdots,x^\nu_m),$$ and $\pi=(\pi_\nu):\prod\Q_\nu^{m_\nu} \to \prod\Q_\nu^{m_\nu-1}$
	and let 
	\begin{equation}
	\widetilde\bV:=\pi(\widetilde R_F\cap\bU_0),
	\\
	\bV=\bigcup_{3\rho(\beta_F)-\text{balls} B\subset \widetilde{\bV}}\frac{1}{2}B
	\end{equation}
	and
	\begin{equation}
	R_F=\left\{\begin{array}{l}
	
	\pi\inv(\bV)\cap\widetilde{R_F}  \ \text{if} \ \ |\partial_1(F^\nu+\Theta^\nu)(\bx_\nu)|_\nu> \lambda_\nu|\nabla(F^{\nu}+\Theta^\nu)(\bx_\nu)|_\nu \ \forall \nu\in S, \ \forall \ \bx=(\bx_\nu)\in \bU_0\\
	
	\emptyset  \ \ \ \ \ \ \ \ \ \ \ \ \ \ \ \  \text{otherwise}.
	\end{array}\right 
	.\end{equation} 
	where $0<\lambda_\nu<1$ is fixed for each $\nu\in S$. \\
	 We claim that the $R_F$ are resonant sets. The intersection property is satisfied for the same reason as it is satisfied in the case $\infty \notin S$. It remains to check the covering property (\ref{coveringproperty}) to establish ubiquity. Without loss of generality we will assume that the ball $\bU_0$ in the definition of (\ref{nice}) satisfies
	\begin{equation}
	\diam(\bU_0)\leq\min_{\nu\neq\infty} (\frac{1}{p_\nu}, \frac{\delta}{(n+1)2nm_\infty\prod_{\nu\in S\setminus\infty}p_\nu^{n+2}})
	\end{equation} where $\delta$ is the one coming in the definition \ref{nice}.
	From the Definition \ref{nice} of  $\f$ being nice at $\bx_0,$ there exist fixed $0<\delta,w<1$ such that for any arbitrary ball $\bB\subset\bU_0,$ 
	\begin{equation}
	\limsup_{Q\to \infty}|\Phi^{\f}(Q,\delta)\cap \frac{1}{2}\bB |\leq w|\frac{1}{2}\bB|.
	\end{equation}
	So for sufficiently large $Q$ we have that 
	$$
	|\frac{1}{2}\bB\setminus \Phi^{\f}(Q,\delta)|\geq \frac{1}{2}(1-w)|\frac{1}{2}\bB|=2^{-(\sum (m_\nu)+1)}(1-w)|\bB|. 
	$$
	Therefore it is enough to show that 
	\begin{equation}
	\frac{1}{2}\bB\setminus \Phi^{\f}(Q,\delta)\subset\bigcup _{F\in\mathcal{F}_n\\
		\beta_F\leq Q}\Delta(R_F,\rho(Q))\cap\bB.
	\end{equation}
	Suppose $\bx\in \frac{1}{2}\bB\setminus \Phi^{\f}(Q,\delta).$ Consider the lattice
	\begin{equation}
	\Gamma_{\bx}=\left\{(a_0,a_1,\cdots,a_n)\in\Z^{n+1}: \begin{array}{l}|a_0+a_1f^\nu_1(\bx)+\cdots+a_nf^\nu_n(\bx)|^l_{\nu}<\delta Q^{-(n)} ~\forall \nu\in S\setminus\infty,\\
	|a_i|_\nu\leq\frac{1}{p_\nu} \ \forall \nu\in S\setminus\infty, \ \forall \ {1\leq i\leq n}\end{array}\right\},
	\end{equation}
	and the convex set  
	\begin{equation}
	K_{\bx}=\left\{(y_0,\by)=(y_0,y_1,\cdots,y_n)\in\R^{n+1}~|~\begin{array}{l}
	\vert y_0+\by\cdot\f^\infty(\bx)\vert_\infty^l<\delta Q^{-n},\\
	\Vert \by\Vert_\infty\leq Q
	\end{array}\right\}.
	\end{equation} 
	
	By Lemma \ref{covolume} we have that 
	\begin{eqnarray*}
\Vol(\R^{n+1}/\Gamma_\bx) & = & \prod_{\nu\in S\setminus\infty} p_\nu^n.p_\nu^{-[\log_{p_\nu}(\delta Q^{-n})^\frac{1}{l}]} \\
    & \leq & \prod_{\nu\in S\setminus\infty} p_\nu^{n}\frac{1}{\prod_{\nu\in S\setminus\infty} p_\nu^{\log_{p_\nu}{(\delta Q^{-n})^\frac{1}{l}-1}}} \\
    &\leq & \left(\frac{Q^n}{\delta}\right)^{\frac{l-1}{l}}.\prod_{\nu\in S\setminus\infty} p_\nu^{(n+1)}.
  \end{eqnarray*}
	Also note that $\Vol(K_\bx)=2^{n+1}(\delta Q^{-n})^{\frac{1}{l}}.Q^n=2^{n+1}\delta^{\frac{1}{l}}Q^{\frac{n(l-1)}{l}}$.
	Using the fact that $\bx\notin \Phi^{\f}(Q,\delta) $ we get the first minima $\lambda_1=\lambda_1(\Gamma_{\bx},K_\bx)>1$. Therefore using Minkowski's theorem on successive minima, we have that $$
	2^{n+1}\delta^{\frac{1}{l}}Q^{\frac{n(l-1)}{l}}\lambda_1.\lambda_2.\cdots.\lambda_{n+1}\leq 2^{n+1}\Vol(\R^{n+1}/\Gamma_{\bx})\leq 2^{n+1}\leq\big(\frac{Q^n}{\delta}\big)^{\frac{l-1}{l}}.\prod_{\nu\in S\setminus\infty} p_\nu^{(n+1)}.
	$$
	This implies that $\lambda_{n+1}\leq \frac{\prod_{\nu\in S\setminus\infty} p_\nu^{n+1}}{\delta}.$ By the definition of $\lambda_{n+1}$ we get $n+1$ linearly independent integer vectors $\ba_j=(a_{j,0},\cdots,a_{j,n})\in\Z^{n+1}(0\leq j\leq n)$ such that the functions $F_j$ given by $$
	F_j(\by)=a_{j,0}+a_{j,1}f_1(\by)+\cdots+a_{j,n}f_n(\by)
	$$ satisfy 
	\begin{equation} \label{conditions_R}
	\left\{ \begin{array}{l}
	|F^\infty_j(\bx)|^l_\infty<\big(\frac{\prod_{\nu\in S\setminus\infty} p_\nu^{n+1}}{\delta}\big)^l\delta Q^{-n}\\
	|F^\nu_j(\bx)|^l_\nu<\delta Q^{-n}\\
	|a_{j,i}|_\infty\leq Q.\frac{\prod_{\nu\in S\setminus\infty} p_\nu^{n+1}}{\delta}\\
	|a_{j,i}|_\nu\leq\frac{1}{p_\nu} \text{ for } 0\leq j \leq n, 1\leq i\leq n\ \forall \nu\in S\setminus \infty.
	\end{array}\right.
	\end{equation} 
	Now consider the following system of linear equations for each $\nu\in S$,
	\begin{equation} \label{linear_R}
	\begin{array}{l}
	\eta_0F_0(\bx)+\eta_1 F_1(\bx)+\cdots+\eta_n F_n(\bx)+\Theta(\bx)=0\\
	\eta_0\partial_1F_0(\bx)+\eta_1\partial_1F_1(\bx)+\cdots+\eta_n\partial_1F_n(\bx)+\partial_1\Theta(\bx)=D\\
	\eta_0 a_{0,j}+\cdots+\eta_n a_{n,j}=0 \ \  (2\leq j \leq n),
	\end{array}
	\end{equation} where $D\in\Q_S$, $$
	\begin{array}{l}
	D_\infty=Q+(\prod_{\nu\in S\setminus\infty}p_\nu)\sum_{i=0}^n\vert \partial_1 F_i^\infty(\bx)\vert_\infty
	\\ D_\nu=1, \forall \nu\in S\setminus\infty.
	\end{array}$$
	Since $f_1(\bx)=x_1 $, the determinant of this aforementioned  system is $\det(a_{j,i})\neq 0$. Therefore there exists a unique solution to the system, say $(\eta_0,\eta_1,\cdots,\eta_n)\in \Q_S^n$. Using the strong approximation Theorem \ref{Strong} for $i=0,\cdots,n$ we get $r_i\in\Q$ such that
	\begin{equation} \label{r_i_R}
	\begin{aligned}
	& |r_i-\eta_i^\infty|_\infty\leq\prod_{\nu\in S\setminus\infty} p_\nu\\
	&
	|r_i-\eta^\nu_i|_\nu\leq 1 \ \forall\nu\in S\setminus\infty,\\
	& 
	|r_i|_{\nu}\leq 1
	\quad \text{for all} \ \text{ prime }\nu\notin S\setminus\infty.
	\end{aligned}
	\end{equation}
	Now take the function 
	\begin{equation}\begin{aligned}
	F(\by)=r_0F_0(\by)+r_1F_1(\by)+\cdots+r_nF_n(\by)\\
	=a_0+a_1f_1(\by)+\cdots+a_nf_n(\by),
	\end{aligned}
	\end{equation} where 
	\begin{equation} \label{a_i_R}
	a_i=r_0a_{0,i}+r_1a_{1,i}+\cdots+r_na_{n,i},
	\  \forall \ i=0,\cdots,n.
	\end{equation}
	
	We claim that\\
	
	\textbf{Claim $1$.}The $a_i$s are all integers.\\ 
	
	From (\ref{r_i_R}) and (\ref{a_i_R}) we get 
	\begin{equation}\label{claim1.1_R}
	|a_i|_\nu\leq 1, \ \forall \ \ {i}=0,\cdots,n \text{ for } \nu\notin S\setminus\infty
	\end{equation}
	and by (\ref{r_i_R}), (\ref{linear_R}) and (\ref{conditions_R}) we have 
	\begin{equation}\begin{aligned}
	|a_i|_\nu\leq \max_{j=0,\cdots,n} \{|\eta^\nu_j-r_j|_\nu|a_{j,i}|_\nu\}\leq 1 \ \forall \nu\in S\setminus\infty
	\quad \text{and for } i=2,\cdots,n.
	\end{aligned}
	\end{equation} 
	Therefore the $a_i's$ are all integers for $i=2,\cdots,n$.
	Now note that $$
	F(\bx)+\Theta(\bx)\\
	=(r_0-\eta_0)F_0(\bx)+\cdots+(r_n-\eta_n)F_n(\bx).
	$$ 
	We denote $F=(F^\nu)$ and $\Theta=(\Theta^\nu) $ and consider $\Q\subset \Q_S$, diagonally embedded. Then we have \begin{equation}\label{condition1}
	|(F+\theta)(\bx)|^l_S\leq T.\delta Q^{-n},
	\end{equation} where $T=\max\big((n+1)^l\left(\frac{\prod_{\nu\in S\setminus\infty} p_\nu^{n+2}}{\delta}\right)^l{\delta}, \delta)$.
	Again 
	$$
	\partial_1(F+\Theta)^\nu(\bx_\nu)=(r_0-\eta^\nu_0)\partial_1F_0(\bx_\nu)+\cdots+(r_n-\eta^\nu_n)\partial_1F_n(\bx_\nu)+D_\nu ~\forall~\nu\in S.
	$$ 
	Since $\forall~\nu\in S\setminus\infty $, 
	$1\leq i\leq n$ and 
	$0\leq j\leq n$, we have $|a_{j,i}|_\nu\leq\frac{1}{p_\nu}$ so that $$|\partial_1F^\nu_j(\bx_\nu)|_\nu\leq \frac{1}{p_\nu}$$ and thus by (\ref{r_i_R}) we get 
	\begin{equation}\label{partial_condition_R}
	1-\frac{1}{p_\nu}\leq|\partial_1(F^\nu+\Theta^\nu)(\bx_\nu)|_\nu\leq 1.
	\end{equation} 
	Also, 
	\begin{equation}\label{partial_condition_R_infty}
	|\partial_1(F+\Theta)^\infty(\bx_\infty)|_\infty\leq C_2\cdot Q,
	\end{equation} where $C_2$ does not depend on $\bx$.
	Now we can show that $a_1$ and $a_0$ are also integers. Since $f_1(\by)=y_1,$ we have 
	\begin{equation}\label{a_1_R}
	a_1=\partial_1(F+\Theta)(\bx)-\partial_1\Theta(\bx)-\sum_{j =2}^{n}a_j\partial_1f_j(\bx)
	\end{equation}
	which implies that $|a_1|_\nu\leq 1 \ \forall~\nu\in S\setminus\infty$. This together with (\ref{claim1.1_R}) proves that  $a_1$ is an integer. We similarly prove that $a_0$ is an integer. We can write
	\begin{equation}\begin{aligned} \label{a_0_R}
	a_0=(F+\Theta)(\bx)-\Theta(\bx)-\sum_{j =1}^{n}a_jf_j(\bx).
	\end{aligned}
	\end{equation}
	This implies that $|a_0|_\nu\leq 1\ \forall~\nu\in S\setminus\infty$ and thus by (\ref{a_0_R}) and (\ref{claim1.1_R}) we get that $a_0$ is integer. So the first claim is proved.
	
	Now we look at the infinity norm of the integers $a_i$.
	By (\ref{a_i_R}), (\ref{linear_R}) and (\ref{r_i_R}) we have 
	\begin{equation}\label{a_infty_R}
	\begin{aligned}
	|a_i|_\infty\leq|(r_0-\eta_0^\infty) a_{0,i}+\cdots+(r_n-\eta_n^\infty)a_{n,i}|_\infty\\
	\leq (n+1)Q.\frac{\prod_{\nu\in S\setminus\infty} p_\nu^{n+2}}{\delta}=C_3\cdot Q
	\end{aligned}
	\quad \text{ for } i=2,\cdots,n.
	\end{equation}
From (\ref{a_1_R}), (\ref{partial_condition_R_infty}) and (\ref{a_infty_R}) we have 
\begin{equation} \label{upperbounda_1}
\vert a_1\vert_\infty\leq C_4\cdot Q\end{equation}
 where $C_4$ does not depend on the point $\bx$.
Also, \begin{equation}
\begin{aligned}
\partial_1(F+\Theta)^\infty(\bx_\infty)= Q+\prod_{\nu\in S\setminus\infty}p_\nu\sum_{k=0}^n\vert\partial_1 F_k^\infty(\bx_\infty)\vert+\sum_{k=0}^n(r_k-\eta_k^\infty)\partial_1F_k^\infty(\bx_\infty)\\
\implies Q\leq \vert\partial_1(F+\Theta)^\infty(\bx_\infty)\vert.
\end{aligned}
\end{equation}
Thus from (\ref{a_1_R}) we have that there exists some $C_5$, independent of $\bx$, such that for sufficiently large $Q$, 
\begin{equation}\label{lowerboundona}
C_5\cdot Q\leq \Vert\ba\Vert_\infty.
\end{equation}
	So by (\ref{a_infty_R}), (\ref{lowerboundona}) and (\ref{upperbounda_1}) we get 
	\begin{equation} \label{beta_R}
\frac{C_5}{\max{(C_3,C_4)}}.Q<\beta_F=\frac{1}{ \max(C_3,C_4)} \Vert\ba\Vert_\infty\leq Q,\end{equation} 
	here $\kappa_0=\frac{1}{\max{(C_3, C_4)}}$, which depends on $\delta$, $n$ and the set $S\setminus\infty$.
	Note that  for all $\by\in\bU_0$ we have 
	\begin{equation}
	\partial_1(F+\Theta)(\bx)=\partial_1(F+\Theta)(\by)+\sum\Phi_{j1}(\partial_1(F+\Theta))(\star)(x_j-y_j)
	\end{equation}
	where $\star$ is from the coefficients of $\bx$ and $\by$. By using (\ref{partial_condition_R}) and by the fact that $$\diam(\bU_0)\leq\min_{\nu\neq\infty} (\frac{1}{p_\nu}, \frac{\delta}{(n+1)2nm_\infty\prod_{\nu\in S\setminus\infty}p_\nu^{n+2}}),$$ we have
	\begin{equation}
	|\partial_1(F^\nu+\Theta^\nu)(\by_\nu)|_\nu\geq 1-\frac{2}{p_\nu} \ \ \forall \  \by_\nu\in\bU^\nu_0 \ \forall~\nu\in S\setminus \infty
	\end{equation}
	and
	\begin{equation}
	|\partial_1(F^\infty+\Theta^\infty)(\by_\infty)|_\infty\geq\frac{1}{2}Q.
	\end{equation}
\noindent	
	So $F$ satisfies $|\partial_1(F^\nu+\Theta^\nu)(\by_\nu)|_\nu> \lambda_\nu|\nabla(F^\nu+\Theta^\nu)(\by_\nu)|_\nu \ \forall \ \by=(\by_\nu)\in \prod_{\nu\in S}\bU^\nu_0=\bU_0$ and thus by construction, $\Delta(R_F,\rho(Q))\neq \emptyset$.\\ 
	
	\textbf{Claim $2$.} $\bx \in \Delta(R_F,\rho(Q))$.\\
	
	We set $r_0 := \diam(\bB)$ and define the function 
	$$g^\nu(\xi^\nu) :=(F^\nu+\Theta^\nu)(x^\nu_1+\xi^\nu,x_2,\cdots,x^\nu_{m_\nu}), \text { where } |\xi|_S=|(\xi^\nu)|_S<r_0$$ and for all $\nu\in S$.
	Then \begin{equation}
	\begin{aligned}
	|g^\nu(0)|_\nu=|(F^\nu+\Theta^\nu)(\bx_\nu)|^l_\nu<T\delta Q^{-n} ~\forall~\nu \in S \\
	\text{ and } |{g^\nu}'(0)|_\nu=|\partial_1(F^\nu+\Theta^\nu)(\bx_\nu)|_\nu>1-\frac{1}{p_\nu}~ \forall~\nu\in S\setminus \infty,\\ \vert{g^\infty}'(0)\vert_\infty=|\partial_1(F^\infty+\Theta^\infty)(\bx_\infty)|_\infty>Q.
	\end{aligned}
	\end{equation} 
	Now applying Newton's method, there exists $\xi^\nu_0$ such that $g^\nu(\xi^\nu_0)=0$ and $|\xi^\nu_0|_\nu<\frac{p_\nu}{(p_\nu-1)}(T\delta Q^{-n})^{\frac{1}{l}}$ for $\nu\in S\setminus\infty$ and by the Lemma 6 of \cite{BaBeVe}, there exists $\xi_0^\infty$ such that  $g^\infty(\xi^\infty_0)=0$ and $\vert\xi^\infty_0\vert_\infty<(T\delta Q^{-n})^\frac{1}{l}Q^{-1}$. For sufficiently large $Q$ we get $\bx_{\xi_0}=(x_1+\xi_0,x_1,\cdots,x_n)\in \bB,$ that $(F+\Theta)(\bx_{\xi_0})=0$ and that $$\Vert\bx-\bx_{\xi_0}\Vert_S\leq\max(\left(\max_{\nu\in S\setminus \infty} \frac{p_\nu}{(p_\nu-1)}\right)(T\delta Q^{-n})^{\frac{1}{l}}, (T\delta Q^{-n})^\frac{1}{l}Q^{-1}).$$ Now we argue exactly as in \cite{BaBeVe}. We recall the argument for the sake of completeness. By  the Mean Value Theorem we have
	$$\begin{aligned}
	|(F+\Theta)(\by)|^l_S \ll Q^{-n}\\
	\text{ for any } \Vert\by-\bx_{\xi_0}\Vert_S \ll Q^{-\frac{n}{l}}.
	\end{aligned}
	$$  
	Then by (\ref{beta_R})  and using the same argument as above tells us that for sufficiently large $Q>0$ the ball of radius $\rho(\beta_F)$ centred at $\pi\bx_{\xi_0}$ is contained in $\widetilde{\bV}$. This ultimately gives $\bx_{\xi_0}\in R_F$ . Since 
	$$\Vert\bx-\bx_{\xi_0}\Vert_S\leq \left(\max_{\nu\in S\setminus\infty} \frac{p_\nu}{(p_\nu-1)}\right)(T\delta Q^{-n})^{\frac{1}{l}}$$ 
	so $\bx\in\Delta(R_F,\rho(Q))$ where $\rho(Q)= \left(\max \frac{p_\nu}{(p_\nu-1)}\right)(T\delta Q^{-n})^{\frac{1}{l}}=\kappa_1Q^{-\frac{n}{l}}$.
	Therefore $\bx\in \Delta(R_F,\rho(Q))$ for some $F\in\mathcal{F}_n $ such that $\beta_F\leq Q$ and this completes the proof of the Theorem. 
\end{proof}
\subsubsection{ Proof of the main divergence theorem}
Using Theorem \ref{ubiquity_R} and Lemma \ref{ubi} we can complete the proof of Theorem \ref{thm:nice_R}.

Fix $\bx_0\in \bU$ and let $\bU_0$ be the neighbourhood of $\bx_0$ which comes from (\ref{ubiquity}). We need to show that 
$$\mathcal{H}^s(\mathcal{W}^\f_{(\Psi,\Theta)}\cap\bU_0)=\mathcal{H}^s(\bU_0)
$$ if the series in (\ref{main sum}) diverges. Consider $\phi(r):=\psi(\kappa_0\inv r)^\frac{1}{l} $. Our first aim is to show that
\begin{equation}
\Lambda(\phi)\subset \mathcal{W}^\f_{(\Psi,\Theta)}.
\end{equation}
Note that $\bx\in \Lambda(\phi)$ implies the existence of infinitely many  $F\in\mathcal{F}_n $ 
such that $\dist(\bx,R_F)<\phi(\beta_F)$. For such $F\in\mathcal{F}_n$ there exists  $\bz\in\bU_0$ such that $(F+\Theta)(\bz)=0$ and $\Vert\bx-\bz\Vert_S<\phi(\beta_F)$. By the mean value theorem for all $\nu\in S\setminus\infty$,
$$
(F+\Theta)^\nu(\bx_\nu)=(F+\Theta)^\nu(\bz_\nu)+ \nabla(F + \Theta)^\nu(\bx_\nu)\cdot (\bx_\nu - \bz_\nu) + \sum_{i,j}\Phi_{ij}(F+\Theta)^\nu(\star)(x_i - z_i)_\nu(x_j-z_j)_\nu, $$
where $\star$ comes from the coefficients of $\bx_\nu$ and $\bz_\nu $. Then we have that 
\begin{equation}
|(F+\Theta)^\nu(\bx_\nu)|_\nu^l\leq\Vert\bx_\nu-\bz_\nu\Vert^l_\nu<\phi(\beta_F)^l=\phi(\kappa_0 \Vert\ba\Vert_\infty)^l=\Psi(\ba) ~\forall~\nu\in S\setminus \infty.
\end{equation}
Similarly for archimedean part we use the mean value theorem and get,
\begin{equation}\begin{aligned}
\vert (F+\Theta)^\infty(\bx_\infty)\vert_\infty= \vert \sum_{i=1}^{m_\infty} \partial_i (F+\Theta)^\infty(\star)(x_i-z_i)_\infty\vert_\infty\\
\leq \Vert \bx_\infty-\bz_\infty\Vert_\infty.\vert \sum_{i=1}^{m_\infty}\partial_i(\sum_{j=1}^n a_jf_j^\infty+\Theta^\infty)(\star)\vert_\infty\\
\leq 2n \Vert \bx_\infty-\bz_\infty\Vert_\infty  \Vert\ba\Vert_\infty\\
\leq \phi(\beta_F)\Vert \ba\Vert_\infty.
\end{aligned}
\end{equation}
Hence $\vert (F+\Theta)^\infty(\bx_\infty)\vert_\infty^l\leq 2n.\Psi(\ba)\Vert\ba\Vert_\infty$, and so
$\Lambda(\phi)\subset \mathcal{W}^\f_{(\Psi,\Theta)} $. Now the Theorem will follow if we can show that
$$\sum_{t=1}^{\infty} \frac{\phi(2^t)^{s-m+l}}{\rho(2^t)^l}=\infty ,$$ where $m=\sum m_\nu$. 
Observe that
$$\sum_{t=1}^{\infty} \frac{\phi(2^t)^{s-m+l}}{\rho(2^t)^l}\asymp \sum_{t=1}^\infty (\psi(\kappa_0\inv 2^t))^{\frac{s-m+l}{l}}\frac{1}{(\rho(2^t))^l}\\
\asymp \sum_{t=1}^\infty (\psi(\kappa_0\inv 2^t))^{\frac{s-m+l}{l}}2^{tn}$$

As $\psi$ is an approximating function so we got that the above series $$\gg\sum_{t=1}^\infty \sum_{\kappa_0\inv 2^t<\Vert\ba\Vert\leq\kappa_0\inv2^{t+1} }(\psi(\Vert\ba\Vert))^{\frac{s-m+l}{l}}\asymp  \sum_{\ba\in\Z^{n}\setminus{0}}(\psi(\Vert\ba\Vert))^{\frac{s-m+l}{l}}\\
$$ $$=\sum_{\ba\in\Z^{n}\setminus{0}}\Psi(\ba)^{\frac{s-m+l}{l}}=\infty.
$$
This completes the proof of the Theorem.

\section{$(C, \alpha)$-good functions}
In this section, we recall the notion of $(C, \alpha)$-good functions on ultrametric spaces. We follow the treatment of Kleinbock and Tomanov \cite{KT}. Let $X$ be a metric space, $\mu$ a Borel measure on $X$ and let $(F, |\cdot|)$ be a local field. For a subset $U$ of $X$ and $C, \alpha > 0$, say that a Borel measurable function $f : U \to F$ is $(C, \alpha)$-good on $U$ with respect to $\mu$ if for any open ball $B \subset U$ centred in $\supp \mu$ and $\varepsilon > 0$ one has
\begin{equation}\label{gooddef}
\mu \left(\{ x \in B \big| |f(x)| < \varepsilon \} \right) \leq
C\left(\displaystyle \frac{\varepsilon}{\sup_{x \in
B}|f(x)|}\right)^{\al}|B|,
\end{equation}
 The following elementary properties of $(C,
\al)$-good functions will be used.
\begin{enumerate}
\item[(G1)] If $f$ is $(C,\al)$-good on an open set $V$, so
is $\lambda
f~\forall~\lambda \in
F$;
\item[(G2)] If $f_i, i \in I$ are $(C,\al)$-good on $V$, so
is $\sup_{i \in
I}|f_i|$;
\item[(G3)] If $f$ is $(C,\al)$-good on $V$ and for some
$c_1,c_2\,\textgreater \,0,\, c_1\leq
\frac{|f(x)|}{|g(x)|}\leq c_2
\text{ for all }x \in V$, then g is
$(C(c_2/c_1)^{\al},\al)$-good on $V$.
\item[(G4)] If $f$ is $(C,\al)$-good on $V$, it is
$(C',\alpha')$-good
on $V'$ for every $C' \geq \max\{C,1\}$, $\alpha' \leq \alpha$ and $V'\subset V$.
\end{enumerate}
One can note that from (G2), it follows that the supremum
norm of a vector valued function $\f$ is $(C,\al)$-good
whenever each of its components is $(C,\al)$-good.
Furthermore, in view of (G3), we can replace the norm by an
equivalent one, only affecting
$C$ but not $\al$.

Polynomials in $d$ variables of degree at most $k$ defined on local fields can be seen to be $(C, 1/dk)$-good, with $C$ depending only on $d$ and $k$ using Lagrange interpolation. In \cite{KM}, \cite{BKM} and \cite{KT} (for ultrametric fields), this property was extended to smooth functions satisfying certain properties. We rapidly recall, following \cite{S} (see also \cite{KT}), the definition of smooth functions in the ultrametric case. Let $U$ be a non-empty subset of $X$ without isolated points. For $n \in \mathbb{N}$, define
$$\nabla^{n}(U) = \{(x_1,\dots,x_n) \in U, x_i \neq x_j \text{ for } i \neq j   \}.$$
The $n$-th order difference quotient of a function $f : U \to X$
is the function  $\Phi_n(f) $ defined inductively by $\Phi_0 (f) =
f$ and, for $n \in \N$,  and $(x_1,\dots,x_{n+1}) \in \nabla^n(U)$
by
\[ \Phi_{n}f(x_1,\dots,x_{n+1}) = \frac{\Phi_{n-1}f(x_1,x_3,\dots,x_{n+1}) -
 \Phi_{n-1}f(x_2,\dots,x_{n+1})}{x_1-x_2}. \]
This definition does not depend on the choice of variables,
as all difference quotients are symmetric functions. A function $f$
on $X$ is called a $C^n$ function if $\Phi_n f$ can be extended to
a continuous function $\bar{\Phi}_{n}f : U^{n+1} \to X $. We also set
\[ D_n f(a) = \overline{\Phi}_nf(a,\dots,a),~a \in U. \]
We have the following theorem (c.f. \cite{S}, Theorem $29.5$).
\begin{theorem}\label{derivative}
Let $f \in C^{n}(U \to X)$. Then, $f$ is $n$ times differentiable
and
\[ j!D_j f = f^j  \]
for all $1 \leq j \leq n$.
\end{theorem}

To define $C^{k}$ functions in several variables, we follow the notation set
forth in \cite{KT}. Consider a multiindex $\beta =
(i_1,\dots,i_d)$ and let
\[ \Phi_{\beta}f = \Phi^{i_1}_{1}\circ \dots \circ \Phi^{i_d}_{d} f.  \]
This difference order quotient is defined on the set $
\nabla^{i_1}U_1 \times \dots \times \nabla^{i_d}U_d$ and the $U_i$
are all non-empty subsets of $X$ without isolated points. A
function $f$ will then be said to belong to $C^{k}(U_1\times \dots
\times U_d)$ if for any multiindex $\beta$ with $|\beta| = \sum_{j =
1}^{d} i_j \leq k$, $\Phi_{\beta} f$ extends to a continuous
function $\bar{\Phi}_{\beta}f : U_{1}^{i_1 + 1} \times \dots \times
U_{d}^{i_d + 1}$. We then have
\begin{equation}\label{multivanish}
\partial_{\beta}f(x_1,\dots,x_d) = \beta!
\bar{\Phi}_{\beta}(x_1,\dots,x_1,\dots,x_d,\dots,x_d)
\end{equation}
where $\beta ! = \prod_{j = 1}^{d} i_{j}!$.\\

We are now ready to gather the results on ultrametric $(C, \alpha)$-good functions that we need. We begin with Theorem $4.1$ from \cite{KT}.
\begin{theorem}\label{theorem 3.2} 
	            Let $V_1,V_2,\cdots,V_3$ be nonempty open sets in F, ultrametric field. Let $ k\in \N$, $A_1,\cdots,A_d> 0 $ and $ f\in C^k(V_1\times\cdots,\times V_n) $ be such that 
	            \begin{equation}\label{eqn 3.3}
	            |\Phi_j^kf|\equiv A_j \text{ on } \nabla^{k+1}V_j\times\prod_{i\neq j}V_i , j=1,\cdots,d.
	            \end{equation}
	            Then f is $(dk^{3-\frac{1}{k}},\frac{1}{dk})$-good on $V_1\times\cdots,\times V_n$
	            \end{theorem}
 The following is an ultrametric analogue of Proposition 1 from \cite{BaBeVe}. This constitutes a generalisation of Theorem 4.5 in \cite{MoS1}. We omit the proof which is a straightforward adaptation of the proof in \cite{BaBeVe}.
  We set $\mathcal{O}=\Z[\frac{1}{p}]$, where $p$ is prime.
               \begin{proposition}\label{Calpha_Prop}
              	Let $U_\nu$ be an open subset of $\Q_\nu ^d,$ $\bx_0 \in U_\nu$ and let 
              	$\mathcal{F}\subset C^l(U)$ be a compact family of functions $f: U\to \Q_\nu $
              	for some $l\geq 2$. Also assume that 
              	\begin{equation}\label{3.4}
              	\inf_{f\in\mathcal{F}}\max_{0<|\beta|\leq l} \ |\partial_{\beta}f(\bx_0)|>0.
              	\end{equation}
              	
              	Then there exists a neighbourhood $V_\nu\subset U_\nu$ of $\bx_0$ and $C, \delta > 0$ satisfying the following property. For any $\Theta\in C^l(U)$ such that 
              	\begin{equation}\label{theta_cond}
              	\sup_{\bx\in U_\nu} \max_{0<|\beta|\leq l} \ |\partial_{\beta}\Theta(\bx_0)|\leq \delta
              	\end{equation} 
              	and for any $f\in \mathcal {F}$ we have that 
              	\begin{enumerate}
              	\item  $f+\Theta $ is $(C,\frac{1}{dl})$-good on $V_\nu$.
              	\item $|\nabla(f+\Theta)|$ is $ \left(C,\frac{1}{m(l-1)}\right)$-good on $V_\nu$
           	 \end{enumerate}
               \end{proposition}

As a Corollary, we have,

\begin{corollary}\label{good_corollary}
	Let $U_\nu$ be an open subset of $\Q_\nu^{d\nu}, \bx_0\in U_\nu$ be fixed and assume that $\f_\nu=(f_\nu^{(1)},f_\nu^{(2)},\dots,f_\nu^{(n)}): 
	U_\nu\to \Q_\nu^n $ satisfies (I2) and (I3) and that  $\Theta_\nu $ satisfies (I5). Then there exists a neighbourhood $V_\nu\subset U_\nu$ of $\bx_0$ and positive constants $ C > 0 $ and $l\in \N$  such that for any $(a_0,\ba)\in \mathcal{O}^{n+1},$ 
	\begin{enumerate}
		\item $a_0+\ba.\f_{\nu}+\Theta_\nu$ is $(C,\frac{1}{d_\nu l})$-good on $V_\nu,$ and 
		\item $|\nabla(\ba.\f_\nu +\Theta_\nu)| $ is $ (C,\frac{1}{d_\nu(l-1)})$-good on $V_\nu$.
		\end{enumerate}
	 \end{corollary}
	
	     \begin{proof}
	For the case $\nu=\infty$, see Corollary $3$ of \cite{BaBeVe} and also \cite{BKM}. So we may assume $\nu\neq \infty.$
	 Let $\mathcal{F}:= \{a_0+\ba.\f_\nu+\Theta_\nu \ |\  (a_0,\ba)\in\mathcal{O}^{n+1}\}$. This is a compact family of functions of $C^l(U_\nu)$ for every $l>0 $ since $\mathcal{O}$ is compact in $\Q_\nu$. Now if this family satisfies condition (\ref{3.4}) for some $l\in \N$, then the conclusion follows from the previous Proposition. Hence we may assume that the family does not satisfy (\ref{3.4}) for every $l\in \N$. Then by the continuity of differential and the compactness of $\mathcal{O}$, there exists $\bc_l\in \mathcal{O}^n$ such that for every  $2 \leq l\in \N $ we have
	 $$ \max_{|\beta|\leq l}|\partial_{ \beta}(\bc_l.f_\nu+\Theta_\nu)(\bx_0)| > 0.$$
		Now this sequence $\{\bc_l\} \in\mathcal{O}^n$ has a convergent subsequence $\{\bc_{l_k}\}$ converging to $\bc \in \mathcal{O}^n$ since $\mathcal{O}^n$ is compact. By taking limits we get that 
		$$|\partial_{ \beta}(\bc.f_\nu+\Theta_\nu)(\bx_0)|=0 \ \forall \ \beta.$$
		However, as each of the $\f_{\nu}$ and $\Theta_\nu$ are analytic on $U_\nu,$ there exists a neighbourhood $V_{\bx_0}$	 of $\bx_0$ such that 
		$$(\bc.f_\nu+\Theta_\nu)(\bx)=u\ \forall \ \bx \in V_{\bx_0},$$ 
		where $u \in \Q_\nu$ is a constant. Therefore replacing  $\Theta_\nu$ by $u-\bc.\f_{\nu},$ we get that  $$\mathcal{F}=\{ a_0+u+(\ba-\bc).\f_{\nu} \ | (a_0,\ba)\in \mathcal{O}^{n+1} \}.$$
		First consider the case where $|a_0+u| < 2|\ba-\bc|,$ then 
		$$\mathcal{F}_1= \left\{\frac{a_0+u}{|\ba-\bc|}+\frac{\ba-\bc}{|\ba-\bc|}.\f_\nu |\  (a_0,\ba)\in\mathcal{O}^{n+1}\right\}$$ 
		is compact in $C^l(U_\nu)$ for every $l\in \N$. Then by linear independence of $1,f_\nu^{(1)},\cdots,f_\nu^{(n)},$ $\mathcal{F}_1$ satisfies (\ref{3.4}) for some $l\in\N$. And then by Proposition \ref{Calpha_Prop}
		 we can conclude that every element in $\mathcal{F}_1$ is $(C,\frac{1}{d_\nu l})$-good on some $V_\nu\subset V_{\bx_0}\subset U_\nu$ together with conclusion (2) of the Corollary above. This also implies $ a_0+u+(\ba-\bc).\f_{\nu} $ are all $(C,\frac{1}{d_\nu l})$ good on $V_\nu$ for all $(a_0,\ba)\in\mathcal{O}^{n+1}$ with $|a_0+u| < 2|\ba-\bc|$. Otherwise $$\sup_{\bx\in V_{\bx_0}}|a_0+u+(\ba-\bc).\f_{\nu}|\leq 3.\inf_{\bx\in V_{\bx_0}}|a_0+u+(\ba-\bc).\f_{\nu}|$$ as $|a_0+u|\geq 2|\ba-\bc| $ and it turns out to be a trivial case. This implies that for $C\geq 3$ and $0<\alpha\leq1$ the aforementioned functions are $(C,\alpha)$-good. 
		 	\end{proof}
 Let us recall the following Corollary from \cite{KT} (Corollary 2.3).
 \begin{corollary}{\label{product_good}}
 	For $ j=1,\cdots,n,$ let $X_j$ be a metric space, $\mu_j$ be a measure on $X_j $. Let $ U_j\subset X_j $ be open, $C_j,\alpha_j >0 $ and let $f$ be a function on $U_1\times\cdots \times U_d$ such that for any $j=1,\cdots d$ and any $x_i\in U_i$ with $i\neq j,$ the function
 	\begin{equation}{\label{fun}}
 	y~~\mapsto f(x_1,\cdots,x_{j-1}, y, x_{j+1},\cdots, x_d)
 	\end{equation}
 	is $(C_j,\alpha_j)$-good on $U_j$ with respect to $\mu_j$. Then $f$ is $(\widetilde{C},\widetilde{\alpha}) $ -good on $U_1\times\cdots\times U_d $ with respect to $\mu_1\times\cdots\times\mu_d,$  where $\widetilde{C}=d,\widetilde{\alpha }$ are computable in terms of $C_j,\alpha_j $. In particular, if each of the functions (\ref{fun}) is $(C,\alpha)$-good on $U_j$ with respect to $\mu_j$, then the conclusion holds with $\widetilde{\alpha}=\frac{\alpha}{d}$ and $\widetilde{C}=dC$.
		 \end{corollary}
	 Now combining Corollary (\ref{good_corollary}) and (\ref{product_good}) we can state the following:
	 
	 \begin{corollary}\label{good_function}
	Let  $\f $ and $\Theta$ be as in Corollary (\ref{good_corollary}) and let $\bx_0\in \bU.$ Then there exists a neighbourhood $\bV\subset\bU $ of $\bx_0$  and  $C>0, k,k_1\in\N$ such that for any $(a_0,\ba)\in\Z^{n+1} $  the following holds:
	 \begin{enumerate}
	 \item 
	$ \bx~\mapsto ~|(a_0+\ba.\f+\Theta )(\bx)|_S\text{ is } (C,\frac{1}{dk})-\text{good on } \bV$,
	 \item 
	 $\bx~\mapsto~\|\nabla(\ba.\f_{\nu}+\Theta_\nu)(\bx_{\nu})\| \text{ is } (C,\frac{1}{dk_1})-\text{ good on } \bV, \forall ~\nu\in S$
	 \end{enumerate} 
	 where $d=\max{d_\nu}$.
	 \end{corollary}

		 \section{Proof of Theorem \ref{thm:main}}
		 We set $\phi(\nu)=\left\{\begin{array}{rl} -\varepsilon & \text{ if } \nu\neq\infty \\
                1-\varepsilon & \text{ if }\nu=\infty 
                \end{array} \right.$
From the definition, it follows that $\cW_{\Psi,\Theta}^{\f}$ admits a description as a limsup set. Namely,
         $$\well =\limsup_{\ba \to \infty}\bW_\f (\ba,\Psi,\Theta)    $$  
        where 
                   $$\bW_\f (\ba,\Psi,\Theta)=\{\bx\in\bU :| a_0+ \ba\cdot \f(\bx)+\Theta(\bx)|_S^l\leq \Psi(\ba) \text{ for some } a_0 \} .$$
                   We may now write                
$$ \bW_{\f}^{\text{large}}(\ba,\Psi,\Theta)= \left \{\bx\in \Wf~:~\|\nabla(\ba .\f_\nu(\bx_\nu)+\Theta_\nu(\bx_\nu))\|>\|\ba\|^{\phi(\nu)} ~\forall~\nu \right\}$$
 where $0<\varepsilon <\frac{1}{4(n+1)l^2},$ is fixed and 
 $$\Wf\setminus\Wfl =\bigcup_{\nu\in S}\bW_{\nu,\f}^{\text{small}}(\ba, \Psi,\Theta)$$
 where $$\bW_{\nu,\f}^{\text{small}}(\ba,\Psi,\Theta)=\left\{\bx\in\Wf :\|\nabla(\ba.\f_\nu(\bx_\nu)+\Theta_\nu(\bx_\nu))\|\leq\|\ba\|^{\phi(\nu)} \right\}.$$ 
 As the set $S$ is finite, we have
$$\well=\cW_{\f}^{\text{large}}(\Psi,\Theta)\bigcup_{\nu\in S}\cW_{\nu,\f}^{\text{small}}(\Psi,\Theta)$$
                 	where $\cW_{\f}^{\text{large}}(\Psi,\Theta)=\welllarge$ and $\cW_{\nu,\f}^{\text{small}}(\Psi,\Theta)=\wellsmallnu.$
                 To prove Theorem \ref{thm:main}, we will show that each of these limsup sets has zero measure.  Namely, the proof is divided into the ``large derivative" case where we will show $|\cW_{\f}^{\text{large}}(\Psi,\Theta)|=0$, and the ``small derivative" case which involves  $|\cW_{\nu,\f}^{\text{small}}(\Psi,\Theta)|=0 \ \forall\ \nu\in S.$

\subsection{The small derivative}
		 We begin by showing that $|\cW_{\nu,\f}^{\text{small}}(\Psi,\Theta)|=0 \ \forall\ \nu\in S$. From the assumed property (I4) of $\Psi$ and convergence of the series $\sum \Psi(\ba)$, it follows that $ \Psi(\ba)<\Psi_0(\ba) :=\prod_{\substack{i=1,\cdots,n \\ a_i\neq 0}}|a_i|\inv.$ Therefore $\cW_{\nu,\f}^{\text{small}}(\Psi,\Theta)\subset \cW_{\nu,\f}^{\text{small}}(\Psi_0,\Theta)$, which means that it is enough to show that $ |\cW_{\nu,\f}^{\text{small}}(\Psi_0,\Theta)|=0 \ \forall\ \nu\in S $. Let us take $\mathcal{A}=\Z\times\Z^n\setminus\{0\} $ and $\bT=\Z_{\geq 0}^n $ and define the function
		 \begin{equation}\label{r_equation}
		 r_\nu (\bt)=\left\{\begin{array}{rl}
		  2^{(|\bt|+1)(1-\varepsilon)} & \text{if } \nu=\infty\\
		 2^{-(|\bt|+1)\varepsilon} & \text{if } \nu\neq\infty
		 \end{array} \right .
		 \end{equation}
		 where $\varepsilon$ is fixed as before. Now we define sets $ I_\bt^\nu(\alpha,\lambda) $ and $H_\bt^\nu(\alpha,\lambda)$ for every $\lambda>0,\bt\in\bT \text{ and } \alpha=(a_0, \ba)\in \mathcal{A} $ as follows:\\
		 \begin{equation}\label{def_I}
		 I_\bt^\nu(\alpha,\lambda)=\left\{ \bx\in\bU:\begin{array}{l}
		|a_0+\ba.\f(\bx)+\Theta(\bx)|_S^l<\lambda\Psi_0(2^\bt)\\
		\|\nabla(\ba.\f_{\nu}(\bx_\nu)+\Theta_\nu(\bx_\nu))\|<\lambda r_\nu(\bt)\\
		2^{t_i}\leq \max{\{1,|a_i|\}}\leq 2^{t_i+1} \ \forall \ 1\leq i\leq n
		\end{array}  
		\right\}
		 \end{equation}
		 and 
		\begin{equation}\label{def_H}
		 H_\bt^\nu(\alpha,\lambda)=\left\{ \bx\in\bU:\begin{array}{l}
		 |a_0+\ba.\f(\bx)|_S^l<2^l\lambda\Psi_0(2^\bt)\\
		\|\nabla(\ba.\f_{\nu}(\bx_\nu))\|<2\lambda r_\nu(\bt)\\
		|a_i|\leq 2^{t_i+2} \ \forall\  1\leq i \leq n 
		\end{array} 
		 \right\}
		\end{equation}
		where $2^\bt=(2^{t_1},\cdots,2^{t_n})$ and $|S|=l$. These give us the functions (\ref{I_fn}) and (\ref{H_fn}) required in the inhomogeneous transference principle. As in (\ref{defH}) and (\ref{deflambda}) we get $H_\bt^\nu(\lambda)$, $I_\bt^\nu(\lambda)$, $\Lambda_H^\nu(\lambda)$ and $\Lambda_I^\nu(\lambda)$.  Now define $\phi_\delta~:~\bT\mapsto \R_{+}$ as $\phi_\delta(\bt):=2^{\delta|\bt|}$ for $\delta\in(0,\frac{\varepsilon}{2}] $. Clearly $\cW_{\nu,\f}^{\text{small}}(\Psi_0,\Theta)\subset \Lambda_I^\nu(\phi_\delta) $ for every $\delta\in(0,\frac{\varepsilon}{2}]$. So to settle Case 2 it is enough to show that
		\begin{equation} \label{Inhomo_set}
		|\Lambda_I^\nu(\phi_\delta)|=0 \text{ for some } \delta\in(0,\frac{\varepsilon}{2}].
		\end{equation}

	The Theorem \ref{<}  is an $S$-adic analogue of Theorem $1.4$ in \cite{BKM} and is proved using nondivergence estimates for certain flows on homogeneous spaces. We will denote the set in the LHS of (\ref{<eqn}) as
	$S(\delta,K_{\nu_1},\cdots,K_{\nu_l},T_1,\cdots,T_n)$ for further reference.
	
	To show (\ref{Inhomo_set}) we want to use the Inhomogeneous transference principle (\ref{transfer}). Assume that  $(H_\nu,I_\nu,\Phi)$ satisfies the intersection property and that the product measure is contracting with respect to $(I_\nu,\Phi)$ where, $\Phi:=\{\phi_\delta : 0\leq \delta <\frac{\varepsilon }{2}\} $. Then by (\ref{transfer}) it is enough to show that
	\begin{equation}\label{homo_condi}
	|\Lambda_H^\nu(\phi_\delta)|=0 \text{ for some } 0<\delta \leq\frac{\varepsilon}{2}.
	\end{equation}
	
	\noindent Note that $$\Lambda_H^\nu(\phi_\delta)=\limsup_{\bt\in\bT} \bigcup_{\alpha\in\mathcal{A}} H_{\bt}^\nu(\alpha,\phi_\delta(\bt)).$$ 
	Using Theorem \ref{<}, we will show that $$\sum |\cup_{\alpha \in \mathcal{A}}H_{\bt}^\nu(\alpha,\phi_\delta(\bt))|<\infty $$ for some $0<\delta<\frac{\varepsilon}{2}$. This, together with Borel-Cantelli will give us $|\Lambda_H^\nu(\phi_\delta)|=0$.\\

\noindent By the definition  \ref{def_H} of $H_{\bt}^\nu(\alpha,\phi_\delta(\bt)),$ we get 
	$$\bigcup_{\alpha\in\mathcal{A}}H_{\bt}^\nu(\alpha,\phi_\delta(\bt))\subset S(2^l\phi_\delta(\bt)\Psi_0(2^{\bt}),1,\cdots ,2.\phi_\delta(\bt)r_\nu(\bt),\cdots,1,2^{t_1+2}, \dots, 2^{t_n +2})
$$	
i.e. here $K_\nu=2\cdot \phi_\delta(\bt)r_\nu(\bt), K_\omega=1,$ where $\omega\neq\nu$ and $T_i=2^{t_i+2}$. 

	\subsection{Case $1$ $(\nu=\infty)$}  Here $r_\infty(\bt)=2^{(1-\varepsilon)(|\bt|+1)}$. So,
	$$
2^l.2^{\delta|\bt|}\Psi_0(2^{\bt}).2.2^{\delta|\bt|}2^{(1-\varepsilon)(|\bt|+1)}.1.\frac{2^{\sum_{1}^n t_i+2}}{2 ^{|\bt|}}=2^{2n+l+2-\varepsilon}.2^{|\bt|(2\delta-\varepsilon)}<1$$ 
for all large $ \bt$ as $2\delta-\varepsilon<0$. So by Theorem \ref{<} we have 
		$$|\bigcup_{\alpha\in\mathcal{A}}H_{\bt}^\infty(\alpha,\phi_\delta(\bt))|\leq E\varepsilon_1^{\alpha_1}|\mathbf{B}|,
$$ where $\varepsilon_1=\max\{2.2^{\frac{\delta|\bt|-\sum t_i}{l}},2^{\frac{2n+l+2-\varepsilon}{l(n+1)}}.2^{\frac{|\bt|(2\delta-\varepsilon)}{l(n+1)}}\}
=2^{\frac{2n+l+2-\varepsilon}{l(n+1)}}.2^{\frac{|\bt|(2\delta-\varepsilon)}{l(n+1)}}$	 for all large $\bt\in \Z_{\geq 0}^n$. 
We note that $\varepsilon_1$ is ultimately the 2nd term in the parenthesis. Because if not then  for infinitely many $\bt$,
$$ 
\frac{\delta|\bt|-\sum t_i}{l}>\frac{|\bt|(2\delta-\varepsilon)}{l(n+1)} + O(1)$$
which implies that 
$$\sum t_i<|\bt| + O(1),
$$ a contradiction. Therefore we have $$|\bigcup_{\alpha\in\mathcal{A}}H_{\bt}^\infty(\alpha,\phi_\delta(\bt))| \ll 2^{-\gamma|\bt|},
$$ where $\gamma=\frac{(\varepsilon-2\delta)}{l(n+1)}\alpha_1>0$. Hence $$\sum_{\bt\in \bT}|\bigcup_{\alpha\in\mathcal{A}}H_{\bt}^\infty(\alpha,\phi_\delta(\bt))|\ll\sum_{\bt\in\bT} 2^{-\gamma|\bt|}<\infty.$$

\subsection{Case $2$ ($\nu\neq\infty$)} The argument proceeds as in Case $1$. In this case, $r_\nu(\bt)=2^{-\varepsilon (|\bt|+1)}$. So,
$$2^l.2^{\delta|\bt|}\Psi_0(2^{\bt}).2.2^{\delta|\bt|}2^{(-\varepsilon)(|\bt|+1)}.1.\frac{2^{\sum_{1}^n t_i+2}}{2 ^{|\bt|}}=2^{2n+l+1-\varepsilon}.2^{|\bt|(2\delta-\varepsilon-1)}<1$$
for large $\bt$ as $2\delta-\varepsilon<0$. Therefore, by Theorem \ref{<} we have 
 $$|\bigcup_{\alpha\in\mathcal{A}}H_{\bt}^\nu(\alpha,\phi_\delta(\bt))|\leq E\varepsilon_1^{\alpha_1}|\mathbf{B}|,$$ 
 where $\varepsilon_1=\max\{2^{\frac{\delta|\bt|-\sum t_i}{l}},2^{\frac{2n+l+1-\varepsilon}{l(n+1)}}.2^{\frac{|\bt|(2\delta-\varepsilon-1)}{l(n+1)}}\}
 =2^{\frac{2n+l+1-\varepsilon}{l(n+1)}}.2^{\frac{|\bt|(2\delta-\varepsilon-1)}{l(n+1)}}$ for all large $\bt\in \Z_{\geq 0}^n$. 
 As in case $1$, $\varepsilon_1$ is ultimately the 2nd term in the parenthesis. For if not, then for infinitely many $\bt$, 
 $$ 
 \frac{\delta|\bt|-\sum t_i}{l}>\frac{|\bt|(2\delta-\varepsilon-1)}{l(n+1)}+ O(1)$$
which implies that 
$$ \sum t_i<2|\bt|+ O(1).$$
 
This gives a contradiction. Therefore we have $$|\bigcup_{\alpha\in\mathcal{A}}H_{\bt}^\nu(\alpha,\phi_\delta(\bt))|\ll 2^{-\gamma|\bt|},$$ where $\gamma=\frac{(\varepsilon-2\delta+1)}{l(n+1)}\alpha_1>0$. Hence $$\sum_{\bt\in \bT}|\bigcup_{\alpha\in\mathcal{A}}H_{\bt}^\nu(\alpha,\phi_\delta(\bt))|\ll \sum_{\bt\in\bT} 2^{-\gamma|\bt|}<\infty.$$
 
  
  \begin{remark}
  	We will consider $|.|$ the measure to be restricted on some bounded open ball $\bV_{\bx_0}$ around $\bx_0\in \bU$. So we get $|\Lambda^\nu_{I}(\phi_\delta)\cap\bV_{\bx_0} |=0$. But because the space is second countable, we eventually get  $|\Lambda^\nu_{I}(\phi_\delta)|=0$. 
  	\end{remark}
  
 \subsection{Inhomogeneous transference principle} 
In this section we state the  inhomogeneous transference principle of Beresnevich and Velani from \cite[Section 5]{BeVe} which will allow 
us to convert our inhomogeneous problem to the homogeneous one. Let $(\Omega, d)$ be a locally compact metric space. Given two countable indexing sets $\mathcal{A}$ and $\mathbf{T}$, let H and I be two maps from $\bT \times \cA \times \R_{+}$ into the set of open subsets of $\Omega$ such that
\begin{equation}\label{H_fn}
 H~:~(t, \alpha, \lambda) \in \bT \times \cA \times \R_{+} \to H_{\mathbf{t}}(\alpha, \lambda) 
\end{equation}
and
\begin{equation}\label{I_fn}
 I~:~ (t, \alpha, \lambda) \in \bT \times \cA \times \R_{+} \to I_{\mathbf{t}}(\alpha, \lambda). 
\end{equation}  
Furthermore, let
\begin{equation}\label{defH}
 H_{\bt} (\lambda) := \bigcup_{\alpha \in \cA} H_{\mathbf{t}}(\alpha, \lambda) \text{ and }  I_{\bt} (\lambda) := \bigcup_{\alpha \in \cA} I_{\mathbf{t}}(\alpha, \lambda).
\end{equation}
Let $\Psi$ denote a set of functions $\psi: \bT \to \R_{+}~:~\bt \to \psi_{\bt}$. For $\psi \in \Psi$, consider the limsup sets
\begin{equation}\label{deflambda}
\Lambda_{H}(\psi) = \limsup_{\bt \in \bT} H_{\bt}(\psi_{\bt}) \text{ and } \Lambda_{I}(\psi) = \limsup_{\bt \in \bT} I_{\bt}(\psi_{\bt}).
\end{equation}
The sets associated with the map $H$ will be called homogeneous sets and those associated with the map $I$, inhomogeneous sets. We now come to two important properties connecting these notions.

\subsection*{The intersection property} The triple $(H, I, \Psi)$ is said to satisfy the intersection property if, for any $\psi \in \Psi$, there exists $\psi^{*} \in \Psi$ such that, for all but finitely many $\bt \in \bT$ and all distinct $\alpha$ and $\alpha'$ in $\cA$, we have that
\begin{equation}\label{inter}
I_{\bt}(\alpha, \psi_{\bt}) \cap  I_{\bt}(\alpha', \psi_{\bt}) \subset H_{\bt}(\psi^{*}_{\bt}).
\end{equation}

\subsection*{The contraction property}  Let $\mu$ be a non-atomic finite doubling measure supported
on a bounded subset $\mathbf{S}$ of $\Omega$. We recall that $\mu$ is doubling if there is a constant $\lambda > 1$ such that, for any ball $B$ with centre in $\bS$, we have
$$\mu(2B) \leq \lambda \mu(B),$$
where, for a ball $B$ of radius $r$, we denote by $cB$ the ball with the same centre and radius $cr$. 
We say that $\mu$ is contracting with respect to $(I, \Psi)$ if, for any
$\psi \in \Psi$, there exists $\psi^{+}\in \Psi$ and a sequence of positive numbers $\{k_{\bt}\}_{\bt \in \bT}$ satisfying
\begin{equation}\label{conv}
\sum_{\bt \in \bT}k_{\bt} < \infty,
\end{equation}
such that, for all but finitely $\bt \in \bT$ and all $\alpha \in \cA$, there exists a collection $C_{\bt, \alpha}$ of balls $B$
centred at $\mathbf{S}$ satisfying the following conditions:
\begin{equation}\label{inter1}
\bS \cap I_{\bt}(\alpha, \psi_{\bt}) \subset \bigcup_{B \in C_{\bt, \alpha}} B
\end{equation}
\begin{equation}\label{inter2}
\bS \cap \bigcup_{B \in C_{\bt, \alpha}} B \subset  I_{\bt}(\alpha, \psi^{+}_{\bt})
\end{equation}
and
\begin{equation}\label{inter3}
\mu(5B \cap  I_{\bt}(\alpha, \psi_{\bt})) \leq k_{\bt} \mu(5B).
\end{equation}
We are now in a position to state Theorem $5$ from \cite{BeVe} 
\begin{theorem}\label{transfer}
Suppose that $(H, I, \Psi)$ satisfies the
intersection property and that $\mu$ is contracting with respect to $(I, \Psi)$. Then
\begin{equation}\label{eq:transfer1}
\mu(\Lambda_{H}(\psi))=0 ~\forall~\psi \in \Psi  \Rightarrow \mu(\Lambda_{I}(\psi)) = 0 ~\forall~\psi \in \Psi.
\end{equation}
\end{theorem}
 
  \subsection{Verifying the intersection property:}
  
 Let $\bt\in\bT$ with $|\bt|> \frac{l}{1-\frac{\varepsilon}{2}}$. We have to show that for $\phi_\delta$ there exists $\phi_\delta^*$ such that for all but finitely many $\bt\in \bT$ and all distinct $\alpha=(a_0,\ba),\alpha'=(a_0',\ba_0')\in\mathcal{A},$ we have that $I_\bt^\nu(\alpha,\phi_\delta(\bt))\cap I_\bt^\nu(\alpha',\phi_\delta(\bt))\subset H_\bt^\nu(\phi_\delta^*(\bt))$. Consider
  $$\bx\in I_\bt^\nu(\alpha,\phi_\delta(\bt))\cap I_\bt^\nu(\alpha',\phi_\delta(\bt)),$$
  then by Definition (\ref{def_I}) we have 
  \begin{equation}\label{eqn_1}\left\{\begin{array}{l}
  |a_0+\ba.\f(\bx)+\Theta(\bx)|_S<{(\phi_\delta(\bt)\Psi_0(2^\bt))}^\frac{1}{l}\\
  \|\nabla(\ba.\f_{\nu}(\bx_\nu)+\Theta_\nu(\bx_\nu))\|<\phi_\delta(\bt) r_\nu(\bt)\end{array}\right.
  \end{equation} and 
   \begin{equation}\label{eqn_2}\left\{\begin{array}{l}
  |a'_0+\ba'.\f(\bx)+\Theta(\bx)|_S<{(\phi_\delta(\bt)\Psi_0(2^\bt))}^\frac{1}{l}\\
  \|\nabla(\ba'.\f_{\nu}(\bx_\nu)+\Theta_\nu(\bx_\nu))\|<\phi_\delta(\bt) r_\nu(\bt)\end{array}\right.
  \end{equation} where 
  $$|a_i|<2^{t_i+1}\text{ for }1\leq i\leq n \text{ and } |a_i'|<2^{t_i+1}\text{ for } 1\leq i\leq n.$$
Now subtracting the respective equations of (\ref{eqn_2}) from (\ref{eqn_1}) we have  $\alpha''=(a_0-a_0',\ba-\ba')$ satisfying the following equations 
\begin{equation}\label{eqn_3}
\left\{\begin{array}{l}
|a''_0+\ba''.\f(\bx)|_S^l<2^l\phi_\delta(\bt)\Psi_0(2^\bt)\\
\|\nabla(\ba''.\f_{\nu}(\bx_\nu))\|<2\phi_\delta(\bt) r_\nu(\bt)\\
|a''_i|\leq 2^{t_i+2} \ \forall\  1\leq i \leq n .
\end{array} \right.
\end{equation} Observe that $\ba''\neq\mathbf{0}$, because otherwise $$1\leq|a_0''|^l<2^l\phi_\delta(\bt)\Psi_0(2^\bt)<2^l.2^{-{(1-\frac{\varepsilon}{2})}|\bt|},$$ which implies that $|\bt|\leq\frac{l}{1-\frac{\varepsilon}{2}}$, which is true for the finitely many $\bt$'s that we are avoiding. Therefore $\alpha''\in\mathcal{A} $ and $\bx\in H_\bt^\nu(\alpha'',\phi_\delta(\bt))$. So here the particular choice of $\phi_\delta^*$ is $\phi_\delta$ itself. This verifies the intersection property.

\subsection{Verifying the Contraction Property :}

Recall that to verify the contraction property we need to verify the following: for any $\phi_\delta\in \Phi $ we need to find $\Phi_\delta^+\in \Phi$ and a sequence of positive numbers $\{k_{\bt}\}_{\bt\in\bT}$ satisfying $$
\sum_{\bt\in\bT}k_{\bt}<\infty $$
such that for all but finitely many $\bt\in\bT$ and all $\alpha\in\mathcal{A},$ there exists a collection $C_{\bt,\alpha}$ of ball $B$ centred at a point in $\mathbf{S}=\bV=\overbar{\bV}$
satisfying (\ref{inter1}), (\ref{inter2}) and (\ref{inter3}).\\
\noindent
Let us consider the open set $5\bV_{\bx_0}$ in Corollary \ref{good_function}. So we have that 
for any $\bt\in \bT $ and $\alpha=(a_0,\ba)\in \mathcal{A}$ 
\begin{equation}
\mathbf{F}^\nu_{\bt,\alpha}(\bx) :~=\max\{\Psi_0\inv(2^\bt)r_\nu(\bt)|a_0+\ba.\f(\bx)+\Theta(\bx)|_S^l,\|\nabla(\ba.\f_{\nu}+\Theta_\nu)(\bx_\nu)\|\}
\end{equation} 
is $(C,\frac{1}{dk})$-good on $5\bV_{\bx_0}$ for some $C>0,k\in\N$ and $d=\max d_\nu$. Using this new function $\mathbf{F}^\nu_{\bt,\alpha},$ we can write the previous inhomogeneous sets as following :\begin{equation}
I_\bt^\nu(\alpha,\phi_\delta(\bt))=\left\{\bx\in\bU:\begin{array}{l}
\mathbf{F}^\nu_{\bt,\alpha}(\bx)<\phi_\delta(\bt)r_\nu(\bt)\\
\\
2^{t_i}\leq\max\{1,|a_i|\}<2^{t_i+1} ~~\forall~ 1\leq i\leq n 
\end{array}\right\}.\end{equation}\label{inhom_new}
We also note that 
$$ I_\bt^\nu(\alpha,\phi_\delta(\bt))\subset I_\bt^\nu(\alpha,\phi^+_\delta(\bt))$$
where $\phi_\delta^+(\bt)=\phi_{\frac{\delta}{2}+\frac{\varepsilon}{4}} (\bt)\geq \phi_\delta(\bt) ~\forall~ \bt\in\bT$. And $\phi_\delta^+(\bt)=\phi_{\frac{\delta}{2}+\frac{\varepsilon}{4}}(\bt)\in\Phi $ because $\frac{\delta}{2}+\frac{\varepsilon}{4}<\frac{\varepsilon}{2} .$
If $I_\bt^\nu(\alpha,\phi_\delta(\bt))=\emptyset$ then it is trivial. So without loss of generality we can assume that $ I_\bt^\nu(\alpha,\phi_\delta(\bt)) \ne \emptyset $. 
Because for every $\phi_\delta \in \Phi $ , $\phi_\delta(\bt) \Psi_0(2^\bt)<2^{-(1-\frac{\varepsilon}{2})|\bt|}$, so in particular \begin{equation}
I_\bt^\nu(\alpha,\phi_\delta^+(\bt))\subset \{\bx\in \bU ~:~|a_0+\ba.\f(\bx)+\Theta(\bx)|_S^l<2^{-(1-\frac{\varepsilon}{2})|\bt|}\}.
\end{equation}
We recall Corollary 4 of \cite{BaBeVe} ,
$$
\inf_{\substack{(\ba, a_0) \in\R^{n+1}\setminus\{0\} \\ \|\ba\| \geq H_0}}\sup_{\bx\in5\bV_{\bx_0}}|a_0+\ba.\f_\infty(\bx_\infty)+\Theta_\infty(\bx_\infty)|_\infty>0.$$ 
Therefore, 
$$\inf_{\substack{(\ba, a_0)\in\R^{n+1}\setminus\{0\} \\  \|\ba\|\geq H_0 }}\sup_{\bx\in5\bV_{\bx_0}}|a_0+\ba.\f(\bx)+\Theta(\bx)|_S > $$ 
$$\inf_{\substack{(\ba, a_0)\in\R^{n+1}\setminus\{0\} \\ \|\ba\| \geq H_0}}\sup_{\bx\in5\bV_{\bx_0}}|a_0+\ba.\f_\infty(\bx_\infty)+\Theta_\infty(\bx_\infty)|_\infty
> 0.$$
 Now by the $(C,\frac{1}{dk})$-good property of the function $|a_0+\ba.\f(\bx)+\Theta(\bx)|_S^l$ on $5\bV_{\bx_0}$ we conclude 
 $$|I_\bt^\nu(\alpha,\phi_\delta^+(\bt))\cap\bV_{\bx_0}|\leq |\{\bx\in \bV_{\bx_0} ~:~|a_0+\ba.\f(\bx)+\Theta(\bx)|_S^l<2^{-(1-\frac{\varepsilon}{2})|\bt|}\}|$$ $$ \ll2^{-(1-\frac{\varepsilon}{2})(\frac{1}{dk})|\bt|}|\bV_{\bx_0}|$$
for all sufficiently large $|\bt|.$ Therefore $\bV_{\bx_0}\not\subset I_{\bt}^\nu(\alpha,\phi^+_\delta(\bt))$ for sufficiently large $|\bt|$ . The measure restricted to $\bV_{\bx_0}$ will be denoted as $|~~|_{\bV_{\bx_0}}$ and thus $\mathbf{S}=\overbar{\bV_{\bx_0}}$. So $\mathbf{S}\cap I_{\bt}^\nu(\alpha,\phi^+_\delta\bt) $ is open and for every $\bx\in \mathbf{S}\cap I_{\bt}^\nu(\alpha,\phi_\delta(\bt) $ there exists a ball $B'(\bx)\subset I_{\bt}^\nu(\alpha,\phi^+_\delta(\bt)).$ 
 So we can find $\kappa\geq 1$ such that the ball $B=B(\bx):=\kappa B'(\bx)$ satisfies $5 B(\bx)\subset 5V_{\bx_0}$ and
 \begin{equation}\label{twosided_inclusion}
 B(\bx)\cap \mathbf{S}\subset I_{\bt}^\nu(\alpha,\phi^+_{\delta}(\bt))\not\supset \bS\cap 5B(\bx)
 \end{equation}
holds for all but finitely many $\bt$ . The second inequality holds because we would otherwise have $\bV_{\bx_0}\subset I_{\bt}^\nu(\alpha,\phi^+_{\delta}(\bt))$, a contradiction. Then take $C_{\bt,\alpha}:=\{B(\bx)~:~ \bx\in \mathbf{S}\cap I_{\bt}^\nu(\alpha,\phi_{\delta}(\bt))\} $. Hence (\ref{inter1}) and (\ref{inter2}) are satisfied. By (\ref{twosided_inclusion}) we have 
\begin{equation}\label{ineq_1}
\sup_{\bx\in 5B}\mathbf{F}_{\bt,\alpha}^\nu(\bx)\geq \sup_{\bx\in 5B\cap S} \mathbf{F}_{\bt,\alpha}^\nu(\bx)\geq \phi_\delta^+(\bt)r_\nu(\bt)
\end{equation} for all but finitely many $\bt$. So in view of the definitions we get 
\begin{equation}\label{ineq_2}
\sup_{\bx\in 5B\cap  I_{\bt}^\nu(\alpha,\phi_{\delta}(\bt)) }\mathbf{F}_{\bt,\alpha}^\nu(\bx)\leq 2^{(\frac{\delta}{2}-\frac{\varepsilon}{4})|\bt| }\phi_\delta^+(\bt)r_\nu(\bt)\leq_{\ref{ineq_1}}2^{(\frac{\delta}{2}-\frac{\varepsilon}{4})|\bt|}\sup_{\bx\in 5B}\mathbf{F}_{\bt,\alpha}^\nu(\bx).
\end{equation}
Therefore for all large $|\bt|$ and $\alpha \in \Z^{n+1}$ we have \begin{equation}\begin{split}
|5B\cap I_{\bt}^\nu(\alpha,\phi_{\delta}(\bt))|\leq_{\ref{ineq_2}} &
|\{ \bx\in 5B~:~\mathbf{F}_{\bt,\alpha}^\nu(\bx)\leq
 2^{(\frac{\delta}{2}-\frac{\varepsilon}{4})|\bt|}
 \sup_{\bx\in 5B}\mathbf{F}_{\bt,\alpha}^\nu(\bx) \} |\\ &\leq  C2^{(\frac{\delta}{2}-\frac{\varepsilon}{4})\frac{1}{dk}|\bt|}|5B|.\end{split}
\end{equation}
Hence finally we conclude \begin{equation}
\begin{split}
|5B\cap I_{\bt}^\nu(\alpha,\phi_{\delta}(\bt))|_{\bV}\leq|5B\cap I_{\bt}^\nu(\alpha,\phi_{\delta}(\bt))|
&
\\\leq C2^{(\frac{\delta}{2}-\frac{\varepsilon}{4})\frac{1}{dk}|\bt|}|5B|&\\ \leq
C_\star C2^{(\frac{\delta}{2}-\frac{\varepsilon}{4})\frac{1}{dk}|\bt|}|5B|_{\bV_{\bx_0}},
\end{split}
\end{equation} since $5B\subset5\bV_{\bx_0}$.  Here we are using that the measure is doubling and the centre of the ball $5B$ is in $\overbar{\bV_{\bx_0}}$. So $C_\star$ is only dependent on $d_\nu$. We choose $k_{\bt}=C_\star C2^{(\frac{\delta}{2}-\frac{\varepsilon}{4})\frac{1}{dk}|\bt|}$ and as $(\frac{\delta}{2}-\frac{\varepsilon}{4})<0$ we also have $\sum k_{\bt}<\infty$ as required in (\ref{conv}). This verifies the contracting property.

\subsection{The large derivative}
In this section, we will show that $|\cW_{\f}^{\text{large}}(\Psi,\Theta)|=0$. Let us recall Theorem 1.2 from \cite{MoS1}.
\begin{theorem}
	Assume that $\mathbf{U}$ satisfies (I1), $\mathbf{f}$ satisfies (I2), (I3)  and $0<\epsilon< \frac{1}{4n|S|^2}.$  Let $\mathcal{A}$ be
	\begin{equation}\left\{\bx\in\bU|~\exists~\ba \in\Z^n, \frac{T_i}{2}\leq~|a_i|<T_i,
	\begin{array}{l}|\langle \ba . \f(\bx) \rangle|^{l}<\delta(\prod_{i} T_i)^{-1}\\
	\|\ba . \nabla \f_{\nu}(\bx_\nu)\|_{\nu}>\|\ba\|^{-\varepsilon},\hspace{2mm}\nu\neq\infty\\
	\|\ba . \nabla \f_{\nu}(\bx_\nu)\|_{\nu}>\|\ba\|^{1-\varepsilon},\hspace{2mm}\nu=\infty
	\end{array}
	\right\}.
	\end{equation} Then
	$|\mathcal{A}|<C \delta\hspace{1mm}|\bU|,$ for large enough $\max (T_i)$ and a universal constant $C$. 
\end{theorem}
 Note that the function $(\f,\Theta):\bU~\mapsto \Q_S^{n+1}$ satisfies the same properties as $\f$. So as a Corollary of the previous theorem we get,
 \begin{corollary}\label{>coro} Let $0<\varepsilon< \frac{1}{4(n+1)|S|^2}$ and $\mathcal{A}_{(T_i)_{1}^n}$ be the set
 	\begin{equation}
 	\bigcup_{\substack{(\ba,1)\in\Z^{n+1}\\\frac{T_i}{2}\leq~|a_i|_{S}<T_i}}\left\{\bx\in\bU~|\\
 	\begin{array}{l}|\langle \ba . \f(\bx)+\Theta(\bx) \rangle|_S^{l}<\delta(\prod_{i=1}^{n} T_i)^{-1}\\
 		\|\nabla( \ba\f_{\nu}(x)+\Theta_\nu(\bx_{\nu}))\|_{\nu}>\|\ba\|^{-\varepsilon},\hspace{2mm}\nu\neq\infty\\
 		\|\nabla (\ba\f_{\nu}(x_\nu)+\Theta_\nu(\bx_{\nu}))\|_{\nu}>\|\ba\|^{1-\varepsilon},\nu=\infty
 	\end{array}
 	\right\}.
 \end{equation} Then
$|\mathcal{A}_{(T_i)_{1}^n}
|<C \delta\hspace{1mm}|\bU|,$ for large enough $\max (T_i)$ and a universal constant $C$. 
 	\end{corollary}
 Now take $T_i=2^{t_i+1}$ and $\delta=2^{\sum_{1}^n t_i+1}\Psi(2^{\bt})$. As $2^{t_i}\leq|a_i| <2^{t_i+1},$ this implies by (\ref{monotone_cond}) that $\Psi(\ba)\geq \Psi(2^{\bt+1})$ and we have using (\ref{>coro}) that 
 \begin{equation}\label{>measure_inq}
 |\bigcup_{2^{t_i}\leq|a_i|_S<2^{t_i+1}}\bW_{\f}^{\text{large}}(\ba,\Psi,\Theta)|<C2^{\sum_{1}^n t_i+1}\Psi(2^{\bt}).
 \end{equation}
 Note that $$\sum\Psi(\ba)\geq\sum\Psi(2^{t_1+1},\cdots,2^{t_n+1})2^{\sum_{1}^n t_i},$$ so the convergence of $\sum\Psi(\ba)$ implies the convergence of the later. Therefore by (\ref{>measure_inq}) and by the Borel-Cantelli lemma we get that almost every point of $\bU$ are in at most finitely many $\bW_{\f}^{\text{large}}(\ba,\Psi,\Theta)$. Hence $|\cW_{\f}^{\text{large}}(\Psi,\Theta)|=0$ completing the proof.

\end{document}